\documentclass[a4paper,10pt]{article}

\usepackage[utf8]{inputenc}
\usepackage{authblk}

\usepackage{amsmath}
\usepackage{amsmath, amsthm, amssymb}
\usepackage{graphicx}
\usepackage{color}
\usepackage{enumerate}
\usepackage{textcomp}
\usepackage{graphicx}
\usepackage{graphicx,hyperref}
\usepackage{hyperref}
\usepackage{natbib}
\usepackage{fancyhdr}
 \usepackage[top=1.5in, bottom=1.25in, left=1.2in, right=1.2in]{geometry}
\usepackage{listings}
\usepackage{verbatim}
\usepackage{natbib}
\usepackage{caption}
\usepackage{subcaption}
\usepackage{titlesec}

\newtheorem{thm}{Theorem}[section]
\newtheorem{lem}{Lemma}[section]
\newtheorem{cor}{Corollary}[section]

\newtheorem{assumption}{Assumption}[section]
\newtheorem{remark}{Remark}[section]

\numberwithin{equation}{section}

\begin{document}

% AUTHOR:  Change "xart" to your insert's file name.
%\input{paper}

\title{\textbf{\LARGE{Asymptotic Properties of Bayes Risk of a General Class of Shrinkage Priors in Multiple Hypothesis Testing Under Sparsity}}}
\author[1]{Prasenjit Ghosh}
\author[2]{Xueying Tang}
\author[2]{Malay Ghosh}
\author[1]{Arijit Chakrabarti}
\affil[1]{Applied Statistics Unit, Indian Statistical Institute, Kolkata, India}
\affil[2]{Department of Statistics, University of Florida, Gainesville, Florida, USA}

\renewcommand\Authands{ and }

\date{}

\maketitle

% \footnote{Submitted to Bayesian Analysis}

\footnotetext[1]{
 Indian Statistical Institute, Kolkata, India
 \href{mailto:prasenjit$\_$r@isical.ac.in}{prasenjit$\_$r@isical.ac.in}
}
\footnotetext[2]{
 University of Florida, Gainesville, Florida, USA
 \href{mailto:xytang@stat.ufl.edu}{xytang@stat.ufl.edu}
}
\footnotetext[2]{
 University of Florida, Gainesville, Florida, USA
 \href{mailto:ghoshm@stat.ufl.edu}{ghoshm@stat.ufl.edu}
}
\footnotetext[1]{
 Indian Statistical Institute, Kolkata, India
 \href{mailto:arc@isical.ac.in}{arc@isical.ac.in}
}

\begin{abstract} Consider the problem of simultaneous testing for the means of independent normal observations. In this paper, we study some asymptotic optimality properties of certain multiple testing rules induced by a general class of one-group shrinkage priors in a Bayesian decision theoretic framework, where the overall loss is taken as the number of misclassified hypotheses. We assume a two-groups normal mixture model for the data and consider the asymptotic framework adopted in \cite{BCFG2011} who introduced the notion of asymptotic Bayes optimality under sparsity in the context of multiple testing. The general class of one-group priors under study is rich enough to include, among others, the families of three parameter beta, generalized double Pareto priors, and in particular the horseshoe, the normal-exponential-gamma and the Strawderman-Berger priors. We establish that within our chosen asymptotic framework, the multiple testing rules under study asymptotically attain the risk of the Bayes Oracle up to a multiplicative factor, with the constant in the risk close to the constant in the Oracle risk. This is similar to a result obtained in \cite{DG2013} for the multiple testing rule based on the horseshoe estimator introduced in \cite{CPS2009, CPS2010}. We further show that under very mild assumption on the underlying sparsity parameter, the induced decision rules based on an empirical Bayes estimate of the corresponding global shrinkage parameter proposed by \cite{PKV2014}, attain the optimal Bayes risk up to the same multiplicative factor asymptotically. We provide a unifying argument applicable for the general class of priors under study. In the process, we settle a conjecture regarding optimality property of the generalized double Pareto priors made in \cite{DG2013}. Our work also shows that the result in \cite{DG2013} can be improved further.
\end{abstract}

\section[Introduction]{Introduction}
Multiple hypothesis testing has become a topic of growing importance in statistics, particularly for the analysis of high-dimensional data. Its application extends over various scientific fields such as genomics, bio-informatics, medicine, economics, finance, just to name a few. For example, in microarray experiments, thousands of tests are performed simultaneously to identify the differentially expressed genes, that is genes whose expression levels are associated with some biological trait of interest. Microarray experiment is just one out of many examples where one needs to analyze sparse high-dimensional data, the main objective being detection of a few signals amidst a large body of noises. Multiple hypothesis testing is one convenient and fruitful approach towards this end. The biggest impetus to research in multiple hypothesis testing came from the classic paper of \cite{BH1995}. Since then the topic has received considerable attention from both frequentists and Bayesians.\vspace{1.5mm}

In this paper, we consider simultaneous testing for means of independent normal observations. Suppose we have $m$ independent observations $X_1 , \cdots , X_m ,$ such that $X_i {\sim} N(\mu_i,\sigma^2),$ for $i=1,\ldots,m$. The unknown parameters $\mu_1, \cdots, \mu_m $ represent the effects under investigation, while $\sigma^2$ is the variance of the random noise. We wish to test $H_{0i}: \mu_i =0 $ against $H_{1i}: \mu_i \neq 0 $, for $i=1,\ldots,m.$ Our focus is on situations when $m$ is large and the fraction of non-zero $\mu_i$'s is small. For each $i$, $\mu_i$ is assumed to be a random variable whose distribution is determined by the latent binary random variable $\nu_i$, where $\nu_i=0$ denotes the event that $H_{0i}$ is true while $\nu_i=1$ corresponds to the event that $H_{0i}$ is false. Here $\nu_i$'s are assumed to be i.i.d $\mbox{Bernoulli}(p)$ random variables, for some $p$ in $(0,1)$. Under $H_{0i}$, $\mu_i = 0$ i.e. $\mu_i \sim \delta_{\{0\}}$, the distribution having mass 1 at 0, while under $H_{1i}$, $\mu_i \neq 0$ and it is assumed to follow a $N(0,\psi^2)$ distribution with $\psi^2 > 0$. Thus
 \begin{equation}\label{TWO_GROUP_MU}
  \mu_i \stackrel{i. i. d.}{\sim} (1-p)\delta_{\{0\}} + p N(0,\psi^2), \mbox{ $i=1,\ldots,m.$}
   \end{equation} 
The marginal distributions of the $X_i$'s are then given by the following two-groups model:
 \begin{equation}\label{TWO_GROUP_X}
  X_i \stackrel{i. i. d.}{\sim} (1-p)N(0,\sigma^2) + p N(0,\sigma^2+\psi^2), \mbox{ $i=1,\ldots,m.$}
   \end{equation} 
Our testing problem is now equivalent to testing simultaneously
\begin{equation}\label{TO_TEST}
 H_{0i}: \nu_i=0  \mbox{ versus } H_{1i}: \nu_i=1 \mbox{ for } i=1,\ldots,m.
\end{equation}
It is assumed that $p$, $\psi^2$ and $\sigma^2$ depend on the number of hypotheses $m$. The parameter $p$ is the theoretical proportion of non-nulls in the population. In sparse situations, where most of the $\mu_i$'s are zero or very small in magnitude, it is natural to assume that $p$ is small and converges to $0$ as the number of hypotheses $m$ tends to infinity. The variance component $\psi^2$ is typically assumed to be large to identify the true signals. Such a model is very natural where one has few potentially large signals among a large pool of noise terms and has been very popular in the literature. See, for example, \citet{MB1988}, for an early use of modeling of this kind in Bayesian variable selection where a uniform prior is used for the absolutely continuous part in place of the normal prior as in (\ref{TWO_GROUP_MU}) above. The two-groups model has the advantage of capturing information across different tests through learning about the common hyperparameters based on information from all the data points. Fully Bayesian approaches towards multiple testing based on the two-groups model by placing hyperpriors on the underlying model parameters are available in the literature, see, for example, \cite{SB2006} and \cite{BGT2008}. Empirical Bayes approaches using the two-groups formulation have been considered, for example, in \cite{Efron2004, Efron2008}, \cite{Storey2007} and \cite{BGT2008}, just to name a few. Under model (\ref{TWO_GROUP_X}) and the usual additive loss function, \cite{BCFG2011} provided conditions under which the optimal Bayes risk (that is, the risk corresponding to the Bayes rule) can be attained asymptotically under sparsity by a multiple testing procedure as the number of tests grows to infinity. They referred to this property as Asymptotic Bayes Optimality under Sparsity (ABOS). In particular, they showed that the procedures of \cite{BH1995} and Bonferroni attain the ABOS property under mild conditions. The optimal Bayes rule is also referred to as a Bayes Oracle in \cite{BGT2008} and \cite{BCFG2011} and will be discussed in detail further in Section 3.\vspace{1.5mm}

In contrast to the above two-groups formulation, there are proposals to model the unknown parameters in sparse situations through hierarchical one-group ``shrinkage'' priors. Such priors can be expressed as scale-mixtures of normals and their use require substantially less computational effort than the two-groups model, especially, in high-dimensional problems as well as in complex parametric frameworks. These priors capture sparsity by assigning large probabilities to means close to zero while at the same time they give non-trivial probabilities to large means. This is achieved by employing two levels of parameters to express the prior variances of the $\mu_i$'s, namely, the ``local shrinkage parameters'', which control the degree of shrinkage at the individual levels, and a ``global shrinkage parameter'', common for all the $\mu_i$'s to cause an overall shrinking effect. If the mixing density corresponding to the local shrinkage parameters is appropriately heavy tailed, the large observations are left almost unshrunk which is often referred to as the ``tail robustness'' property. Choice of the global shrinkage parameter varies in different specifications and will be discussed in greater detail in Section 2.\vspace{1.5mm}

Some early examples of one-group shrinkage priors are the $t$-prior (\cite{Tipping2001}), the Laplace prior in the context of Bayesian Lasso (\cite{PC2008} and \cite{Hans2009}) and the family of normal-exponential-gamma priors (\cite{GB2005}). More recently, \cite{CPS2009, CPS2010} introduced a hierarchical Bayesian one-group prior called the horseshoe prior. Various new one-group shrinkage priors have been proposed in the literature and studied since then. \cite{ADC2011} introduced the class of ``three parameter beta normal'' mixture priors while the class of ``generalized double Pareto'' priors was introduced in \cite{ADL2012}. The family of three parameter beta normal mixture priors generalizes some well known shrinkage priors such as the horseshoe, Strawderman-Berger and normal-exponential-gamma priors. See also, \cite{PS2011, PS2012}, \cite{Scott2011} and \cite{GB2010, GB2012, GB2013}, in this context. Many of these one-group priors, including the horseshoe, employ local shrinkage parameters with priors having the aforesaid tail robustness property.\vspace{1.5mm}

The horseshoe prior has acquired an important place in the literature on ``shrinkage'' priors and it has been used in estimation as well as in multiple testing and variable selection problems. \citet{CPS2010} proposed a new multiple testing procedure for the normal means problem based on the horseshoe prior. They observed through numerical findings that under sparsity of the true normal means, the procedure based on the horseshoe prior performs closely to the Bayes rule when the true data comes from a two-groups model and the loss of a testing procedure is taken as the number of misclassified hypotheses. \citet{DG2013} theoretically established this optimality by showing that the ratio of the Bayes risk for this procedure to that of the Bayes Oracle under the two-groups model (\ref{TWO_GROUP_X}) is within a constant factor asymptotically. Moreover, it was numerically shown in their paper, that priors having exponential or lighter tails, such as the Laplace or the normal prior, fail to achieve such optimality property.\vspace{1.5mm}

As commented in \cite{CPS2009}, a carefully chosen two-groups model can be considered a ``gold standard'' for sparse problems. Therefore, it may be used as a benchmark against which the ``shrinkage'' priors can be judged. Motivated by this and inspired by the results in \cite{CPS2010} and \cite{DG2013}, we want to study in this paper asymptotic optimality properties of multiple testing procedures induced by a very general class of ``shrinkage'' priors which are heavy tailed and yet handle sparsity well. This class contains the ``three parameter beta normal'' mixture priors as well as the ``generalized double Pareto'' priors. We consider multiple testing rules based on these priors and apply them on data generated from a two-groups model. We establish that these rules achieve the same Bayesian optimality property as shown in \citet{DG2013} for the testing rule based on the horseshoe prior, assuming that the global shrinkage parameter is appropriately chosen based on the theoretical proportion of true alternatives. In case this proportion is unknown, we consider an empirical Bayes version of this test procedure, where the global shrinkage parameter is estimated using the data as in \cite{PKV2014}. We show that the resulting empirical Bayes testing procedure also attains the optimal Bayes risk asymptotically up to the same multiplicative factor. We also study the performance of such rules on simulated data and our theoretical results are corroborated by the simulations.\vspace{1.5mm}

The highlight of this paper is a unified treatment of the question of Bayesian optimality in multiple testing under sparsity based on a very general class of one-group priors, taking the same loss function as in \cite{DG2013}. In the process, we not only generalize their results for a very broad class of tail robust shrinkage priors, but also strengthen their optimality result by deriving a sharper asymptotic upper bound to the corresponding Bayes risk. We have a new unifying argument that enables us to establish asymptotic bounds to the risk for this whole class of priors. \cite{DG2013} conjectured that for the present multiple testing problem, the generalized double Pareto prior should enjoy similar optimality property like the horseshoe prior. We settle this conjecture by showing that the generalized double Pareto is indeed a member of this general class of tail robust priors under consideration. Further, our general technique of proof shows that some of the arguments in \cite{DG2013} can be simplified.\vspace{1.5mm}

The organization of this paper is as follows. In Section 2, we describe the general class of one-group priors under study and define the multiple testing procedure based on them. In Section 3, we present our main theoretical results after describing the optimal Bayes rule under the two-groups model and the asymptotic framework under which these theoretical results are derived. The main results of Section 3 crucially depend on some key inequalities involving the posterior distribution of the underlying shrinkage coefficients, that help us in deriving important asymptotic bounds to the type I and type II error probabilities. These inequalities and the bounds on both types of error probabilities are presented in Section 4. Section 5 contains the simulation results followed by a discussion in Section 6. Proofs of the theoretical results are given in the Appendix.

\subsection{Notations and Definition}
Given any two sequences of positive real numbers $\{a_m\}$ and $\{b_m\}$, with $b_m \neq 0$, we write $a_m \sim b_m $ to denote $\lim_{m \rightarrow \infty} a_m/b_m=1$. For any two sequences of real numbers $\{a_m\}$ and $\{b_m\}$, with $b_m \neq 0$, we write $a_m=O(b_m)$ if $|\frac{a_m}{b_m}| \leq M $ for all $m$, for some positive real number $M$ independent of $m$, and $a_m=o(b_m)$ to denote $\lim_{m\rightarrow\infty}a_m/b_m=0$. Thus $a_m=o(1)$ if $\lim_{m\rightarrow\infty}a_m=0$. Moreover, given any two positive real valued functions $f(x)$ and $g(x)$, both having a common domain of definition $(A,\infty)$, $A \geq 0$, we write $f(x) \sim g(x)$ as $x \rightarrow \infty$ to denote $\lim_{x \rightarrow \infty} f(x)/g(x)=1$.\vspace{1.5mm}

By a random variable $Z$ we mean a $N(0,1)$ random variable having cumulative distribution function and probability density function $\Phi(\cdot)$ and $\phi(\cdot)$, respectively.

\paragraph{Definition}
A positive measurable function $L$ defined over some $(A,\infty)$, $A \geq 0$, is said to be slowly varying or is said to vary slowly (in Karamata's sense) if for every fixed $ \alpha > 0$, $L(\alpha x)\sim L(x)$ as $x \rightarrow \infty$.

% A simple sufficient condition for a positive measurable function $L$ to be slowly varying is that $L(x) \rightarrow c \mbox{ as } x \rightarrow \infty$, for some $c \in (0,\infty)$. Some well known examples of slowly varying functions are $L(t)\equiv c $ for some $c \in (0,\infty)$, $L(t)=\log t$ for $t\in (1,\infty)$, $L(t)=\log\log t$ for $t\in (e,\infty)$ and $L(t)=\big(\frac{t}{1+t}\big)^{\alpha}$ for $t\in (0,\infty)$ with some fixed $\alpha>0$, etc.

% %  Some well known examples of slowly varying functions are $L(t)\equiv c$ for some $c >0$, $L(t)=\log(t)$, $L(t)=1/\log(t)$, $L(t)=\log(\log(t))$. If $L \neq 0$ slowly varies then so does $1/L$. If $L_1$ and $L_2$ are slowly varying functions then so are $L_1+L_2$, $L_1\cdot L_2$ and $L_1 \circ L_2$.
% We consider two approaches in our study. In the former, we treat $\tau$ as a tuning parameter that we are free to choose, while for the latter we use an empirical Bayes estimate of $\tau$. 
\section{The one-group priors and the corresponding induced multiple testing procedures} 
As mentioned in the introduction, our aim in this paper is to study, through theoretical investigations and simulations, asymptotic risk properties of the multiple testing rules induced by a very broad class of one-group shrinkage priors, when applied to data that come from the two-groups model in (\ref{TWO_GROUP_X}). The class of one-group priors we study is inspired by a class of priors suggested in \citet{PS2011} which can be represented through the following hierarchical formulation:
  \begin{eqnarray}
  %\begin{array}{ll}
%   X_i|(\mu_i,\sigma^2) \mbox{ } \sim N(\mu_i,\sigma^2), \mbox{ independently for } i=1,\ldots,m\\
  \mu_i|(\lambda_i^2, \tau^2, \sigma^2) \mbox{ } &\sim& N(0,\lambda_i^2\tau^2\sigma^2), \mbox{ independently for } i=1,\ldots,m, \nonumber\\
  \lambda_i^2 \mbox{ } &\sim& \pi(\lambda_i^2), \mbox{ independently for } i=1,\ldots,m, \mbox{ and}, \nonumber\\
  (\tau^2,\sigma^2) \mbox{ } &\sim& \pi(\tau^2,\sigma^2) \nonumber. 
  %\end{array}
 \end{eqnarray}
Our specific choices of $\pi(\lambda_i^2)$ and $\pi(\tau^2,\sigma^2)$ are described and explained below. Note that the above hierarchy is in slight variation from that of \citet{PS2011} in that we bring the $\sigma$ earlier in the sequence, while in their formulation the $\sigma$ comes later through the conditional prior of $\tau$ given $\sigma$. But both formulations produce the same marginal prior distribution for the $\mu_i$'s.\vspace{1.5mm} 

The above one-group formulation is often referred to as a global-local scale mixtures of normals. The parameter $\tau$ is called a ``global'' shrinkage parameter, while the parameters $\lambda_i^2$'s are called the ``local'' shrinkage parameters. The corresponding posterior mean of $\mu_i$ is given by,
   \begin{equation} \label{POST_MEAN_OG}
   E(\mu_i|X_i,\tau,\sigma)=(1-E(\kappa_i|X_i,\tau,\sigma))X_i,
   \end{equation}
where $\kappa_i=1/(1+\lambda_i^2\tau^2)$ is called the $i$-th shrinkage coefficient. It is observed in \cite{CPS2010} and \cite{PS2011} that under the two-groups model (\ref{TWO_GROUP_X}), for large $\psi^2$, the posterior mean of $\mu_i$ can be approximated as
  \begin{equation}\label{POST_MEAN}
   E(\mu_i|X_i,p,\psi,\sigma) \approx \omega_i(X_i) X_i
  \end{equation}
where $\omega_i(X_i)$ denotes the posterior probability that $H_{1i}$ is true. It may be noted further that when $p \approx 0$, most of the $\omega_i$'s are expected to be very close to zero unless $X_i$ is sufficiently large, in which case the corresponding $\omega_i$ is expected to be close to 1, provided $\psi^2$ is large enough. This ensures that the noise observations are mostly shrunk towards zero, while the large $X_i$'s are left mostly unshrunk. Here the parameter $p$ is responsible for achieving an overall shrinkage, while the large $\psi^2$ is helpful in discovering the true signals.\vspace{1.5mm}

Using the above observations, for the one-group model, \cite{PS2011} argued that in sparse problems, the global shrinkage parameter $\tau$ (whose role is analogous to $p$ in the two-groups prior) should be small and its prior should have substantial mass near zero, whereas the prior for the local shrinkage parameters $\lambda_i^2$ should have thick tails. This ensures that the resulting prior for the $\mu$'s is highly peaked near zero but also heavy tailed enough to accommodate large signals. In this sense, the one-group priors can be thought of as approximately similar to a two-groups prior with an appropriately heavy-tailed absolutely continuous part.\vspace{1.5mm}

Motivated by the preceding discussion and the work of \cite{PS2011}, we take $\pi(\lambda_i^2)$ to be of the form 
  \begin{equation}\label{LAMBDA_PRIOR}
  \pi(\lambda_i^2)=K(\lambda_i^2)^{-a-1}L(\lambda_i^2),
    \end{equation}
in our hierarchical formulation. Here $K > 0$ is the constant of proportionality, $a$ is a positive real number and $L$ is a positive measurable, non-constant, slowly varying function over $(0,\infty)$. It follows from Theorem 1 of \cite{PS2011} that the above general class of one-group priors achieves the desired ``tail robustness'' property in the sense that for any given $\tau$ and $\sigma$, $E(\mu_i|X_i,\tau, \sigma) \approx X_i,$ for large $X_i$'s. Since $\pi(\lambda_i^2)$ is assumed to be proper, the possibility of $L(\cdot)$ being a constant function is ruled out.\vspace{1.5mm}

It will be proved in Section 2.2 and Section 2.3 that a very broad class of one-group priors, such as, the generalized double Pareto and the three parameter beta normal mixtures, actually fall inside the general class of shrinkage priors under consideration. It is worth pointing out here that in some one-group formulations, like the original form of the generalized double Pareto in \cite{ADL2012}, the global shrinkage parameter is not explicitly mentioned or equivalently it is kept fixed at 1. In some other cases, like the three parameter beta normal mixtures, a shared global shrinkage parameter is explicitly given. \cite{ADC2011} opined that it is reasonable to put a prior on the global shrinkage parameter, but this parameter may also be kept fixed at a certain value which reflects the prior knowledge about sparsity {\it if} such information is available. In case such prior knowledge is unavailable, one can consider either of the two approaches, namely, (i) a full Bayes approach by placing further hyperprior over $\tau$ and (ii) an empirical Bayes approach by learning about $\tau$ through the data. In the line of recommendation of \cite{PS2011}, for a full Bayes treatment of the present multiple testing problem, we consider the following joint prior distribution of $(\tau,\sigma)$,
 \begin{equation}\label{TAU_SIGMA_PRIOR}
   \tau \sim  C^{+}(0,1) \mbox{ and } \pi(\sigma) \mbox{ } \propto \frac{1}{\sigma}. 
 \end{equation}
which will be used later in our simulation study. It should be noted here that \cite{GEL2006} strongly recommended the use of a half-Cauchy (or more generally, a folded non-central-$t$) distribution as a prior for the global variance component $\tau$ in a hierarchical Bayesian formulation. Though, in his original recommendation, he suggested using a half-Cauchy prior $C^{+}(0,\sigma)$ for $\tau$ scaled by the error variance $\sigma^2$, we take $C^{+}(0,1)$ since the error variance term $\sigma^2$ appear earlier in our hierarchical formulation. We also consider an empirical Bayes approach to be discussed in detail shortly.\vspace{1.5mm}

We describe below the multiple testing rules considered in this paper. We first consider two rules (defined in (\ref{INDUCED_DECISION}) and (\ref{INDUCED_DECISION_EB}) below) for which asymptotic optimality results have been derived theoretically. For this we assume $\sigma^2$ to be known and equal to 1. For the first rule, we treat $\tau$ as a tuning parameter to be chosen freely depending on the value of $p$, while the second one is based on an empirical Bayes estimate of $\tau$. Note that a comparison between the expressions in (\ref{POST_MEAN_OG}) and (\ref{POST_MEAN}) for the posterior mean of $\mu_i$, together with the previous discussion, suggest that the posterior shrinkage weights $E(1-\kappa_i|X_i,\tau)$ based on tail robust shrinkage priors, should behave like the posterior inclusion probability $\omega_i(X_i)$ in the two-groups model. Using this observation, \cite{CPS2010} proposed a natural classification rule based on the posterior shrinkage weights under a symmetric 0-1 loss for the horseshoe prior. Borrowing the same idea, we consider the following multiple testing procedure based on our chosen class of tail-robust one-group shrinkage priors, given by:
  \begin{equation}\label{INDUCED_DECISION}
  \mbox{ reject } H_{0i} \mbox{ if } 1-E(\kappa_i|X_i,\tau) > 0.5, \mbox{ $i=1,\ldots,m$}.
 \end{equation}  

As mentioned in the introduction, \cite{DG2013} considered the multiple testing rule defined in (\ref{INDUCED_DECISION}) based on the horseshoe prior. They showed that it asymptotically attains the optimal Bayes risk up to a multiplicative factor. It will be seen later that the Oracle optimality property of the decision rule in (\ref{INDUCED_DECISION}) based on our general class of one-group priors, critically depends on appropriate choice of $\tau$ depending on $p$. This plays a significant role in the limiting value of the type II error measure and in controlling the rate of the overall contribution from type I error in the risk function. This is similar to the observations made in \cite{DG2013} for the above multiple testing rule based on the horseshoe prior.\vspace{1.5mm}

In a recent article, \cite{PKV2014} considered the problem of estimating an $m$-dimensional multivariate normal mean vector which is sparse in the nearly black sense, that is, the number of non-zero entries is of a smaller order than $m$ as $m\rightarrow\infty$. They modeled the mean vector through the horseshoe prior and estimated it by the corresponding posterior mean, namely, the horseshoe estimator. They showed that for suitably chosen $\tau$ depending on the proportion of non-zero elements of the mean vector, the horseshoe estimator asymptotically attains the corresponding minimax $l_2$ risk, possibly up to a multiplicative constant, and the corresponding posterior distribution contracts at this optimal rate. But in practice $p$ is usually unknown. A natural approach in such situations is to learn about $\tau$ from the data and then plug this choice into the corresponding posterior mean. When $p$ is unknown, \cite{PKV2014} proposed a natural estimator of $\tau$ and showed that the horseshoe estimator based on this estimate, attains the corresponding minimax $l_2$ risk up to some multiplicative factor. Inspired by this, we consider the following estimator of $\tau$ due to \cite{PKV2014} in case $p$ is unknown:
\begin{equation}\label{TAU_HAT_EMPB}
 \widehat{\tau}= \max\bigg\{\frac{1}{m}, \frac{1}{c_2m}\sum_{j=1}^{m}1\{|X_j|> \sqrt{c_1\log m}\}\bigg\}
\end{equation}
where $c_1 \geq 2$ and $c_2 \geq 1$ are some predetermined finite positive constants. Note that the above estimator of $\tau$ is truncated below by $\frac{1}{m}$ and hence it is not susceptible to collapsing to zero, which is a major concern for the use of such empirical Bayes approaches as mentioned in \cite{CPS2009}, \cite{SB2010}, \cite{BGT2008} and \cite{DG2013}. We refer to \cite{PKV2014} for a detailed discussion on this point. Let $E(1-\kappa_i|X_i, \widehat{\tau})$ denote the posterior shrinkage weight $E(1-\kappa_i|X_i,\tau)$ evaluated at $\tau=\widehat{\tau}$. We consider the following empirical Bayes procedure based on $E(1-\kappa_i|X_i, \widehat{\tau})$, $i=1,\ldots,m$, given by,
\begin{equation}\label{INDUCED_DECISION_EB}
\mbox{ reject } H_{0i} \mbox{ if } 1-E(\kappa_i|X_i,\widehat{\tau}) > 0.5, \mbox{ $i=1,\ldots,m$}.
\end{equation}

For the simulations we consider two cases, firstly, a full Bayes treatment with $(\tau,\sigma)$ given the joint prior distribution as in (\ref{TAU_SIGMA_PRIOR}), and the corresponding rule is defined as
\begin{equation}\label{INDUCED_DECISION_FB}
  \mbox{ reject } H_{0i} \mbox{ if } 1-E(\kappa_i|X_1,\cdots,X_m) > 0.5, \mbox{ $i=1,\ldots,m$}.
\end{equation}
where $1-E(\kappa_i|X_1,\cdots,X_m)$ denotes the $i$-th posterior shrinkage weight after integrating $E(1-\kappa_i|X_i,\tau,\sigma)$ with respect to the joint posterior density of $(\tau,\sigma)$. We also consider the empirical Bayes decisions as in (\ref{INDUCED_DECISION_EB}) for the simulation study where we fix $\sigma^2=1$. We apply the above decision rules in (\ref{INDUCED_DECISION_FB}) or (\ref{INDUCED_DECISION_EB}) induced by these priors in the multiple testing problem (\ref{TO_TEST}), where the true data are generated from the two-groups mixture model (\ref{TWO_GROUP_X}) described before. We show in this paper, through theoretical analysis and simulations that the aforesaid decision rules enjoy similar optimality property as shown for the horseshoe prior in \cite{DG2013}.

\subsection{Some well known one-group shrinkage priors}
In this section, we demonstrate that some popular shrinkage priors actually fall within the general class of one-group priors considered in this paper. This follows from observing that the mixing density $\pi(\lambda_i^2)$ corresponding to the local shrinkage parameter $\lambda_i^2$ can be expressed in the form (\ref{LAMBDA_PRIOR}) where $L(\cdot)$ is a slowly varying function over $(0,\infty)$. This in turn can be shown by proving that the corresponding $L(t)$ converges to a finite positive limit as $t$ goes to infinity. We also show that for each of these priors the corresponding $L(\cdot)$ is uniformly bounded by some finite positive constant. The boundedness property of $L(\cdot)$ is important, as it makes the proofs of the theoretical results of this paper much simpler. This will become clear in Section 4 of this paper.

\subsection{Three Parameter Beta Normal Mixtures}
 Let us consider the following global-local scale mixture formulation of one-group priors:
%  consider the following form of Three Parameter Beta Normal mixture priors, proposed by \cite{ADC2011}:
 \begin{eqnarray}
 \mu_i \mid \lambda_i^2, \tau^2,\sigma^2 \mbox{ } &\sim& \mbox{ } N(0,\lambda_i^2\tau^2\sigma^2)\mbox{ independently for }i=1,\ldots,m\nonumber \\
 \lambda_i^2 \mbox{ } &\sim& \mbox{ } \pi(\lambda_i^2) \mbox{ independently for }i=1,\ldots,m\nonumber\\
 (\tau^2,\sigma^2) &\sim& \pi(\tau^2,\sigma^2)\nonumber
 \end{eqnarray}
 with
 \begin{equation}
 \pi(\lambda_i^2)= \frac{\Gamma(\alpha+\beta)}{\Gamma(\alpha)\Gamma(\beta)}(\lambda_i^2)^{\alpha-1}\big(1+\lambda_i^2\big)^{-(\alpha+\beta)} \label{TPB}
 \end{equation}
 for $\alpha > 0,$ $\beta>0$. The mixing density given in (\ref{TPB}), in fact, corresponds to an inverted-beta density (or, beta density of the second kind) with parameters $\alpha$ and $\beta$. The prior density corresponding to the shrinkage coefficients $\kappa_i=\frac{1}{1+\lambda_i^2\tau^2}$ is then given by,
 \begin{eqnarray}
  \pi(\kappa_i)
%   &=& \frac{\Gamma(\alpha+\beta)}{\Gamma(\alpha)\Gamma(\beta)}\bigg(\frac{1}{\tau^2}\big(\frac{1}{\kappa_i}-1\big)\bigg)^{\alpha-1} \bigg(1+\frac{1}{\tau^2}\big(\frac{1}{\kappa_i}-1\big)\bigg)^{-(\alpha+\beta)}\frac{1}{\tau^2\kappa_i^2}\nonumber\\
  &=& \frac{\Gamma(\alpha+\beta)}{\Gamma(\alpha)\Gamma(\beta)} (\tau^2)^{\beta}\kappa_i^{\beta-1}(1-\kappa_i)^{\alpha-1} \bigg \{1-\big(1-\tau^2\big) \kappa_i \bigg\}^{-(\alpha+\beta)}\nonumber
 \end{eqnarray}
 which corresponds to an $TPB(\alpha,\beta,\tau^2)$ density. Therefore, the above hierarchical one-group formulation can alternatively be represented as
 \begin{eqnarray}
 \mu_i \mid \kappa_i,\sigma^2 \mbox{ } &\sim& \mbox{ } N\big(0,(\kappa_i^{-1} - 1 )\sigma^2 \big)\mbox{ independently for}i=1,\ldots,m\nonumber\\
 \kappa_i \mbox{ } &\sim& \mbox{ } TPB(\alpha,\beta,\tau^2) \mbox{ independently for }i=1,\ldots,m\nonumber
%  (\tau^2,\sigma^2) &\sim& \pi(\tau^2,\sigma^2)\nonumber.
 \end{eqnarray}
 This gives the three parameter beta normal mixture priors introduced by \cite{ADC2011} and is denoted by $TPBN(\alpha,\beta,\tau^2\sigma^2)$. The TPBN family of priors is rich enough to generalize some well known shrinkage priors, such as the horseshoe prior with $\alpha=\frac{1}{2}$, $\beta =\frac{1}{2},$ the Strawderman-Berger prior with $\alpha=1$, $\beta=\frac{1}{2}$ and $\tau^2=1$ and the normal-exponential-gamma priors with $\alpha=1$, $\beta > 0$.\vspace{1.5mm}
  
 Note that the prior in (\ref{TPB}) can also be written as,
 \begin{eqnarray}
  \pi(\lambda_i^2) &=& K(\lambda_i^2)^{-\beta-1} L(\lambda_i^2)\nonumber
  \end{eqnarray}
 where $L(\lambda_i^2)=\big(1+\frac{1}{\lambda_i^2}\big)^{-(\alpha+\beta)}$ and $ K= \frac{\Gamma(\alpha+\beta)}{\Gamma(\alpha)\Gamma(\beta)}$. Clearly, $\lim_{\lambda^2 \rightarrow \infty} L(\lambda^2)= 1$, thereby implying that the TPBN family of priors falls within our general class of global-scale mixture normals. Also, note that $\sup_{t\in(0,\infty)} L(t)=1,$ which shows that the associated function $L(\cdot)$ is bounded as mentioned earlier.
 \subsection{Generalized Double Pareto Priors}
 Let us consider the hierarchical one-group global-local scale mixture formulation as described at the beginning of Section 2, where $\pi(\lambda_i^2)$ is defined as,
  \begin{eqnarray} 
 \lambda_i^2 | \gamma_i \mbox{ } &\sim& \mbox{Exponential}(\frac{\gamma_i^{2}}{2}) \mbox{ independently for } i=1,\ldots,m\nonumber \\
 \gamma_i | \alpha,\beta \mbox{ } &\sim& \mbox{Gamma}(\alpha,\beta) \mbox{ independently for } i=1,\ldots,m.\nonumber
 \end{eqnarray}
 for some fixed $\alpha > 0$ and $\beta > 0$. It follows that $\mu_i \mid (\tau, \sigma) $ has the density
  \begin{equation} \label{GDP_DENSITY}
  \pi(\mu_i|\tau,\sigma)=\frac{1}{2\tau\sigma\beta/\alpha}\bigg(1+\frac{|\mu_i|}{\alpha \cdot \tau\sigma\beta/\alpha} \bigg)^{-(1+\alpha)}
 \end{equation}
 The density in (\ref{GDP_DENSITY}) above, corresponds to a generalized double Pareto density with shape parameter $\alpha$ and scale parameter $\xi=\tau\sigma\beta/\alpha > 0$ and is denoted by $GDP(\alpha,\xi)$. Equivalently, it may also be interpreted as the density of a $GDP(\alpha,\beta/\alpha)$ random variable multiplied by $\tau\sigma.$ When $\alpha=1$ and $\beta=1,$ a $GDP(\alpha,\beta/\alpha)$ distribution is known as the standard double Pareto distribution. We refer to this hierarchical global-local scale mixture formulation with $\pi(\lambda_i^2)$ defined as above, as the generalized double Pareto prior introduced by \cite{ADL2012}. For simulations in Section 5, in our hierarchical global-local scale mixture formulation, when we talk about the standard double Pareto prior, we mean that $\lambda_i^2 \sim GDP(1,1)$, and we mix further with respect to the joint density of $(\tau, \sigma)$ for a full Bayes treatment or use an empirical Bayes estimate of $\tau$ taking $\sigma^2$ to be fixed, as mentioned before.\vspace{1.5mm}
 
 Now we demonstrate that the generalized double Pareto prior falls within our chosen class of tail robust shrinkage priors. Towards this end, we first observe that the mixing density $\pi(\lambda_i^2) $ corresponding to the generalized double Pareto prior can be written as,
  \begin{equation}\label{MIXING_DENSITY_GDP}
  \pi(\lambda_i^2)=\frac{\beta^{\alpha}}{2\Gamma(\alpha)}\int_{0}^{\infty}e^{-(\frac{\gamma_i^2\lambda_i^2}{2}+\beta\gamma_i)}\gamma_i^{\alpha+1}d\gamma_i.
 \end{equation}
Note that using Fubini's Theorem one has $\int_{0}^{\infty} \pi(\lambda_i^2)d\lambda_i^2=1$, so that the density given in (\ref{MIXING_DENSITY_GDP}) is proper. Now, using the change of variable $u= \lambda_i^2\gamma_i^2/2$ in the integral on the right hand side of (\ref{MIXING_DENSITY_GDP}), we obtain,
 \begin{eqnarray}
  \pi(\lambda_i^2)
%   &=& \frac{\beta^{\alpha}}{2\Gamma(\alpha)}\frac{1}{\sqrt{2\lambda_i^2}} \int_{0}^{\infty} \exp\bigg(-\big(u+\beta\sqrt{\frac{2u}{\lambda_i^2}}\big)\bigg) \bigg(\frac{2u}{\lambda_i^2}\bigg)^{\frac{\alpha+1}{2}}\frac{1}{\sqrt{u}}du\nonumber\\
  &=& \frac{\beta^{\alpha}(\lambda_i^2)^{-\frac{\alpha}{2}-1}}{2^{1-\frac{\alpha}{2}}\Gamma(\alpha)} \int_{0}^{\infty} e^{-\beta\sqrt{\frac{2u}{\lambda_i^2}}}e^{-u}u^{(\frac{\alpha}{2}+1)-1}du\nonumber\\
  &=& K(\lambda_i^2)^{-\frac{\alpha}{2}-1} L(\lambda_i^2), \mbox{ say,}\nonumber
  \end{eqnarray}
 where $L(\lambda_i^2)=2^{\frac{\alpha}{2} -1} \int_{0}^{\infty} e^{-\beta\sqrt{\frac{2u}{\lambda_i^2}}}e^{-u}u^{(\frac{\alpha}{2}+1)-1}du$ and $K = \frac{\beta^{\alpha}}{\Gamma(\alpha)}$.\vspace{1.5mm}

Now applying Lebesgue's Dominated Convergence Theorem, we obtain,
  \begin{equation}
   \lim_{\lambda_i^2 \rightarrow \infty} L(\lambda_i^2)= 2^{\frac{\alpha}{2} -1}\int_{0}^{\infty}e^{-u}u^{(\frac{\alpha}{2}+1)-1}du =2^{\frac{\alpha}{2} -1} \Gamma(\frac{\alpha}{2}+1)>0,\nonumber
  \end{equation}
 which means that $L(\cdot)$ defined above slowly varies over $(0,\infty).$ This also shows that the mixing density given in (\ref{MIXING_DENSITY_GDP}) can be expressed in the form given by equation (\ref{LAMBDA_PRIOR}) with $L(\cdot)$ as above and $a=\alpha/2$. Thus, the generalized double Pareto prior falls within our general class of tail-robust shrinkage priors. Moreover, using the monotone convergence theorem, it follows that $\sup_{t\in(0,\infty)} L(t) =2^{\frac{\alpha}{2} -1} \Gamma(\frac{\alpha}{2}+1)$, which means the function $L(\cdot)$ defined above, is bounded as well.\vspace{1.5mm}

%  We conclude this section by mentioning that using similar arguments, one can easily establish that several other popular shrinkage priors, such as the inverse gamma or the half-t priors, also belong to the general class of tail robust shrinkage priors studied in this paper.
%  
\section{Asymptotic framework and the main results}
In this section we present our major theoretical results about asymptotic optimality of the multiple testing rules (\ref{INDUCED_DECISION}) and (\ref{INDUCED_DECISION_EB}) under study. In Section 3.1, first we describe the decision theoretic setting and the optimal Bayes rule under this setting. We then describe the asymptotic framework under which our theoretical results are derived. Section 3.2 presents the main theoretical results of this paper involving asymptotic bounds to the Bayes risk of the induced decisions (\ref{INDUCED_DECISION}) and (\ref{INDUCED_DECISION_EB}) under study. The Oracle optimality properties of these decision rules up to $O(1)$ then follow immediately.

\subsection{Optimal Bayes Rule and the Asymptotic Framework}
 Suppose $X_1,\cdots,X_m$ are independently distributed according to the two-groups model (\ref{TWO_GROUP_X}), with $\sigma^2=1$. We are interested in the multiple testing problem (\ref{TO_TEST}). We assume a symmetric 0-1 loss for each individual test and the total loss of a multiple testing procedure is assumed to be the sum of the individual losses incurred in each test. Letting $t_{1i}$ and $t_{2i}$ denote the probabilities of type I and type II errors respectively of the $i$-th test, the Bayes risk of a multiple testing procedure under the two-groups model (\ref{TWO_GROUP_X}) is given by   
 \begin{equation}\label{BAYES_RISK_GEN}
  R=\sum \limits_{i=1}^{m} \big\{(1-p)t_{1i}+pt_{2i}\big\}.
 \end{equation}
It was shown in \cite{BGT2008} and \cite{BCFG2011} that the multiple testing rule which minimizes the Bayes Risk in (\ref{BAYES_RISK_GEN}) is the test which, for each $i=1,\ldots,m,$ rejects $H_{0i}$ if
 \begin{equation}
  \frac{f(x_i|\nu_i=1)}{f(x_i|\nu_i=0)} > \frac{1-p}{p}, \mbox{ i.e. }X_i^2 > c^2.\nonumber
 \end{equation}
 where $f(x_i|\nu_i=1)$ denotes the marginal density of $X_i$ under $H_{1i}$ while $f(x_i|\nu_i=0)$ denotes that under $H_{0i}$ and $c^2\equiv c^2_{\psi,f}=\frac{1+\psi^2}{\psi^2}(\log(1+\psi^2)+ 2\log(f))$, with $f= \frac{1-p}{p}$.
% \begin{equation}
% c^2\equiv c^2_{\psi,f}=\frac{1+\psi^2}{\psi^2}(\log(1+\psi^2)+ 2\log(f)) \mbox{ with } f= \frac{1-p}{p}.\nonumber
% \end{equation}
The above rule is called Bayes Oracle since it makes use of the unknown parameters $\psi$ and $p$, and hence is not attainable in finite samples. By introducing two new parameters $u=\psi^2 $ and $v=uf^2, $ the above threshold becomes
 \begin{equation}
  c^2\equiv c^2_{u,v}=(1+\frac{1}{u})(\log v + \log (1+\frac{1}{u})).\nonumber
 \end{equation}
 \cite{BCFG2011} considered the following asymptotic scheme:
 \begin{assumption}\label{ASSUMPTION_ASYMP}
     The sequence of vectors $(\psi_m,p_m)$ satisfies the following conditions:
  \begin{enumerate}
    \item $p_m \rightarrow 0 \mbox{ as }m \rightarrow \infty$.
    
    \item $u_m=\psi_m^2 \rightarrow \infty \mbox{ as }m \rightarrow \infty$.
    
    \item $v_m=u_mf^2=\psi_m^2{(\frac{1-p_m}{p_m})}^2 \rightarrow \infty \mbox{ as }m \rightarrow \infty$.
    
    \item $\frac{\log{v_m}}{u_m} \rightarrow C \in (0,\infty) \mbox{ as }m \rightarrow \infty$.
  \end{enumerate}
 \end{assumption}
 
 Under Assumption \ref{ASSUMPTION_ASYMP}, \cite{BCFG2011} obtained the following asymptotic expressions of type I and type II error probabilities of the Bayes Oracle, given by,
   \begin{eqnarray}
    t_1^{BO} &=& e^{-C/2} \sqrt{\frac{2}{\pi v \log v}}(1+o(1)), \mbox{ and }\label{T1_OPT}\\
    t_2^{BO} &=& (2\Phi(\sqrt{C})-1)(1+o(1)),\label{T2_OPT}
   \end{eqnarray}
%    where the $o(1)$ terms above tend to zero as $m \rightarrow \infty$.
 and the corresponding optimal Bayes risk is given by,
   \begin{equation}\label{OPT_BAYES_RISK}
    R_{Opt}^{BO}=m((1-p)t_1^{BO}+pt_2^{BO})=mp(2\Phi(\sqrt{C})-1)(1+o(1)).
   \end{equation}
In (\ref{T1_OPT})-(\ref{OPT_BAYES_RISK}) above, the $o(1)$ terms tend to zero as $m\rightarrow\infty$.\vspace{1.5mm}
 
We want to study asymptotic optimality properties of the multiple testing rules (\ref{INDUCED_DECISION}) and (\ref{INDUCED_DECISION_EB}), induced by our general class of one-group tail robust shrinkage priors when applied to data generated from the two-groups model (\ref{TWO_GROUP_X}), where the hyperparameters $(\psi_m, p_m)$ of the two-groups model satisfy Assumption \ref{ASSUMPTION_ASYMP}. For simplicity of notation, henceforth we drop the subscript $m$ from $p_m$, $\tau^2_m$ and $\psi^2_m$. For the sake of completeness, we describe below the one-group prior specification for our theoretical analysis:
 \begin{equation}\label{FULL_OG_MODEL}
 \begin{array}{ll}
%  X_i|\mu_i \mbox{ } &\stackrel{ind}{\sim} N(\mu_i,1), \mbox{ for } i=1,\ldots,m\\
 \mu_i|(\lambda_i^2, \tau^2) &\stackrel{ind}{\sim} N(0,\lambda_i^2\tau^2), \mbox{ for } i=1,\ldots,m,\\
  \lambda_i^2 &\stackrel{ind}{\sim} \pi(\lambda_i^2) = K (\lambda_i^2)^{-a-1} L(\lambda_i^2), \mbox{ for } i=1,\ldots,m,
 \end{array} \bigg\rbrace 
 \end{equation}
  where $a> 0$, $K >0$ and $L$ is a non-constant slowly varying function over $(0, \infty)$. Under (\ref{FULL_OG_MODEL}), the shrinkage coefficients $\kappa_i=1/(1+\lambda_i^2\tau^2)$'s are independently distributed given $(X_1,\cdots,X_m,\tau^2)$, with the posterior of $\kappa_i$ only depending on $(X_i, \tau^2)$ and is given by
  \begin{equation}
   \pi(\kappa_i|X_i,\tau) \propto \kappa_i^{a+\frac{1}{2}-1}(1-\kappa_i)^{-a-1}L\big(\frac{1}{\tau^2}\big(\frac{1}{\kappa_i}-1\big)\big) e^{-\frac{\kappa_i X_{i}^{2}}{2}} , \mbox{ $\kappa_i \in (0,1)$}.\nonumber
  \end{equation}
 \subsection{Main Theoretical Results}
 In this section, we present in Theorem \ref{THM_BAYES_RISK_UBLB} and Theorem \ref{THM_BAYES_RISK_EB_UB} the main theoretical findings of this paper. Theorem \ref{THM_BAYES_RISK_UBLB} gives asymptotic upper and lower bounds to the Bayes risk of the multiple testing procedure (\ref{INDUCED_DECISION}) under study, when the global shrinkage parameter $\tau$ is treated as a tuning parameter, while Theorem \ref{THM_BAYES_RISK_EB_UB} gives asymptotic upper bounds to the Bayes risk of the empirical Bayes procedure defined in (\ref{INDUCED_DECISION_EB}). Proofs of Theorem \ref{THM_BAYES_RISK_UBLB} and Theorem \ref{THM_BAYES_RISK_EB_UB} are based on some asymptotic bounds for the corresponding type I and type II error probabilities of the individual decisions in (\ref{INDUCED_DECISION}) and (\ref{INDUCED_DECISION_EB}), which, in turn, depend on a set of concentration and moment inequalities. We present these inequalities and the asymptotic bounds on both kinds of error probabilities in Section 4 of this paper. Proofs of Theorem \ref{THM_BAYES_RISK_UBLB} and Theorem \ref{THM_BAYES_RISK_EB_UB} are given in the Appendix.
 
\begin{thm}\label{THM_BAYES_RISK_UBLB}
 Suppose $X_1,\cdots,X_m$, are i.i.d. observations having the two-groups normal mixture distribution in (\ref{TWO_GROUP_X}) with $\sigma^2=1$, and we wish to test the $m$ hypotheses $H_{0i}:\nu_i=0$ vs $H_{1i}:\nu_i=1$, for $i=1,\ldots,m$, simultaneously, using the decision rule (\ref{INDUCED_DECISION}) induced by the one-group priors (\ref{FULL_OG_MODEL}). Suppose Assumption \ref{ASSUMPTION_ASYMP} is satisfied by the sequence of parameters $(\psi^2,p)$. Further assume that $\tau \rightarrow 0$ as $m\rightarrow\infty$ such that $\lim_{m\rightarrow\infty} \tau/p \in(0,\infty)$, and $\pi(\lambda_i^2)$ is such that
  \begin{enumerate}[(I)]
   \item $\frac{1}{2} < a < 1$
   \item  $a= \frac{1}{2}$ and $L(t)/\sqrt{\log(t)} \rightarrow 0$ as $t \rightarrow \infty$.
   \end{enumerate} 
 Then, as $m \rightarrow \infty$, the Bayes risk of the multiple testing rules in (\ref{INDUCED_DECISION}), denoted $R_{OG}$, satisfies
  \begin{equation}\label{UB_LB_BAYES_RISK_OG}
   mp\big[2\Phi\big(\sqrt{2a}\sqrt{C}\big)-1\big]\big(1+o(1)\big)
   \leq R_{OG}
   \leq mp\big[2\Phi\bigg(\sqrt{\frac{2aC}{\eta(1-\delta)}}\bigg)-1\big]\big(1+o(1)\big)
  \end{equation}
 for every fixed $\eta \in (0,\frac{1}{2})$ and $\delta \in (0,1)$. The $o(1)$ terms above are not necessarily the same, tend to zero as $m\rightarrow\infty$ and depend on the choice of $\eta \in (0,\frac{1}{2})$ and $\delta \in (0,1)$.
  \end{thm}

 As a consequence of Theorem \ref{THM_BAYES_RISK_UBLB}, for a very large class of priors covered by (I) or (II) of the said theorem, the ratio of the Bayes risk of the induced decisions in (\ref{INDUCED_DECISION}) to that of the Bayes Oracle (see (\ref{OPT_BAYES_RISK}) in Section 3.1) is asymptotically bounded by,
  \begin{equation}\label{RATIO_RISK}
   \frac{2\Phi(\sqrt{2a}\sqrt{C})-1}{2\Phi(\sqrt{C})-1}\big(1+o(1)\big) \leq \frac{R_{OG}}{R_{Opt}^{BO}} \leq \frac{2\Phi(\sqrt{2a/(\eta(1-\delta))}\sqrt{C})-1}{2\Phi(\sqrt{C})-1}\big(1+o(1)\big) \mbox{ as } m \rightarrow \infty,
  \end{equation}
  for every fixed $\eta \in (0, \frac{1}{2})$ and every fixed $\delta \in (0,1)$. That is,
   \begin{equation}
    R_{OG}=O(R_{Opt}^{BO}) \mbox{ as } m \rightarrow \infty.\nonumber
   \end{equation}
 For small values of $C$ and appropriately chosen $\eta \in (0,\frac{1}{2})$ and $\delta \in (0,1),$ the ratios in (\ref{RATIO_RISK}) given above, can be made close to 1. Therefore, we see that, in sparse situations, when the global shrinkage parameter $\tau$ is asymptotically of the same order as that of the proportion of true alternatives $p,$ the decision rules (\ref{INDUCED_DECISION}), imposed by a very broad class of tail robust one-group priors satisfying (I) or (II) of Theorem \ref{THM_BAYES_RISK_UBLB}, asymptotically attain the optimal Bayes risk up to a multiplicative constant, the constant being close to 1. It may be seen that the condition (II) of Theorem \ref{THM_BAYES_RISK_UBLB} is satisfied if, in the prior on the local shrinkage parameter in (\ref{FULL_OG_MODEL}), one has $a=\frac{1}{2}$ and $L(\cdot)$ is, say, uniformly bounded or $\lim_{t \rightarrow \infty} L(t) \in (0, \infty)$. It has already been shown in Section 2 that the horseshoe prior, the Strawderman-Berger prior and members from the families of normal-exponential-gamma priors and generalized double Pareto priors with appropriate choice of $(\alpha, \beta)$, satisfy these conditions.\vspace{1.5mm}

 The theoretical results of the forthcoming sections of this paper suggest that, for the above Oracle optimality property to be true, the optimal choice of $\tau$ is such that it is asymptotically of the same order of $p$, that is, $\frac{\tau}{p}$ has a finite, positive limit as the number of tests $m$ grows to infinity. It will be shown further that there are other choices of $\tau$ depending on $p$, for which the desired Oracle optimality up to $O(1)$ may no longer be true. These will be discussed later in a greater detail in Section 4.2 of this paper.\vspace{1.5mm}
 
 The next theorem gives an asymptotic upper bound for the Bayes risk of the empirical Bayes procedure defined in (\ref{INDUCED_DECISION_EB}) under the asymptotic framework of \cite{BCFG2011} together with the assumption that $p\equiv p_m \propto m^{-\epsilon}$ for $0 < \epsilon < 1$. As a consequence, the Oracle optimality property of the empirical Bayes procedure (\ref{INDUCED_DECISION_EB}) follows immediately. Note that, the condition $p\propto m^{-\epsilon}$, where $0 < \epsilon < 1$, is very mild in nature and covers most of the cases of theoretical and practical interest.
\begin{thm}\label{THM_BAYES_RISK_EB_UB}
 Suppose $X_1,\cdots,X_m$, are i.i.d. observations having the two-groups mixture distribution described in (\ref{TWO_GROUP_X}) with $\sigma^2=1$, and we wish to test the $m$ hypotheses $H_{0i}:\nu_i=0$ vs $H_{1i}:\nu_i=1$, $i=1,\ldots,m$, simultaneously, using the decision rule (\ref{INDUCED_DECISION_EB}) induced by the one-group priors (\ref{FULL_OG_MODEL}). Suppose Assumption \ref{ASSUMPTION_ASYMP} is satisfied by $(\psi^2, p)$ with $p \propto m^{-\epsilon}$, for some $0 < \epsilon < 1$. Further assume that in the prior $\pi(\lambda_i^2)$ for the local shrinkage parameter $\lambda_i^2$ in (\ref{FULL_OG_MODEL}) satisfies:
  \begin{enumerate}[(I)]
   \item $\frac{1}{2} < a < 1$, or,
   \item  $a= \frac{1}{2}$ and $L(t)/\sqrt{\log(t)} \rightarrow 0$ as $t \rightarrow \infty$.
   \end{enumerate}
 Then, the Bayes risk of the multiple testing rules in (\ref{INDUCED_DECISION_EB}), denoted $R^{EB}_{OG}$, is bounded above by,
  \begin{equation}\label{UB_EMP_BAYES_RISK}
%    mp\bigg(2\Phi\big(\sqrt{2a}\sqrt{C}\big)-1\bigg)\big(1+o(1)\big)
   R^{EB}_{OG}
   \leq mp\big[2\Phi\bigg(\sqrt{\frac{2aC}{\eta(1-\delta)}}\bigg)-1\big]\big(1+o(1)\big) \mbox{ as }m\rightarrow\infty,
 \end{equation}
 for every fixed $\eta \in (0,\frac{1}{2})$ and $\delta \in (0,1)$, where the $o(1)$ term above tends to zero as $m\rightarrow\infty$ and depends on the choice of $\eta \in (0,\frac{1}{2})$ and $\delta \in (0,1)$.
  \end{thm}

Now using Theorem \ref{THM_BAYES_RISK_EB_UB} it follows immediately that
   \begin{equation}
    R^{EB}_{OG}=O(R_{Opt}^{BO}) \mbox{ as } m \rightarrow \infty.\nonumber
   \end{equation}
 As before, for small values of $C$, and appropriately chosen $\eta \in (0,\frac{1}{2})$ and $\delta \in (0,1)$, the ratio of risk $R^{EB}_{OG}/R_{Opt}^{BO}$ can be made close to 1.\vspace{1.5mm}
 
 Using the techniques employed for deriving asymptotic upper bounds for the type I and type II error probabilities of the empirical Bayes decisions in (\ref{INDUCED_DECISION_EB}), one can show easily that the empirical Bayes estimate $\widehat{\tau}$ defined in (\ref{TAU_HAT_EMPB}), consistently estimates the unknown degree of sparsity $p$ up to some multiplicative factor. This will be made more precise in Remark \ref{REM_TAU_CONV}. As mentioned already that the desired Bayesian optimality property as presented in Theorem \ref{THM_BAYES_RISK_UBLB} holds when $\tau$ is asymptotically of the same order of $p$, which seems to be an optimal choice of $\tau$ in case $p$ is known. This perhaps explains the good performance of our proposed empirical Bayes procedure using the estimate $\widehat{\tau}$ and gives a strong theoretical support in favor of using such a plug-in estimate of $\tau$.
 
 \subsection{A comparison with the work of Datta and Ghosh (2013)}
 A careful inspection of the proof of Theorem 3.4 of \cite{DG2013} reveals the following. Under Assumption \ref{ASSUMPTION_ASYMP}, when $\lim_{m\rightarrow\infty} \tau/p \in(0,\infty)$, the Bayes risk of the decision rules (\ref{INDUCED_DECISION}) induced by the horseshoe prior, denoted $R_{HS(DG)}$, satisfies,
  \begin{equation}\label{UB_DG}
   R_{HS(DG)} \leq mp\big[2\Phi\bigg(\sqrt{\frac{2C}{\eta(1-\delta)}}\bigg)-1\big]\big(1+o(1)\big) \mbox { as } m \rightarrow \infty,
  \end{equation}
 for every fixed $\eta \in (0,\frac{1}{2})$ and $\delta \in (0,1)$. A comparison between the upper bounds in (\ref{UB_LB_BAYES_RISK_OG}) and (\ref{UB_DG}) shows that our results not only generalize the theoretical finding concerning the asymptotic Bayes optimality of the horseshoe prior, but at the same time, sharpens the upper bound to the Bayes risk of the induced decisions under study, for $\frac{1}{2} \leq a < 1,$ across the general class of priors given in (\ref{FULL_OG_MODEL}), and satisfying conditions (I) or (II) of Theorem \ref{THM_BAYES_RISK_UBLB}, including the horseshoe, in particular.\vspace{1.5mm}
  
 Although a few ideas employed in the proofs of this paper are similar to those in \cite{DG2013}, our arguments heavily hinge upon appropriate use of properties of slowly varying functions. It will be observed later in this paper that application of well-known properties of slowly varying functions often leads to exact asymptotic orders of certain integrals, without the need to depend mainly on using algebraic upper and lower bounds which can be improved further. In fact, using this technique, we obtain a sharper asymptotic bound to the probability of type II errors and hence on the overall risk (in Theorem \ref{THM_BAYES_RISK_UBLB}) as compared to that in \cite{DG2013}. See Remark 4.2 in this context.
  
 \section{Some key inequalities and bounds on probabilities of type I and type II errors}
In Section 4.1, we present some concentration and moment inequalities involving the posterior distributions of the shrinkage coefficients $\kappa_i$'s. These inequalities are essential for deriving asymptotic bounds for probabilities of type I and type II errors of the multiple testing procedures (\ref{INDUCED_DECISION}) and (\ref{INDUCED_DECISION_EB}) under study, presented in Section 4.2 and Section 4.3, respectively. Proofs of all these results are given in the Appendix.
 \subsection{Concentration and Moment Inequalities}
Before presenting the theoretical results of this section, let us first briefly describe how they can be useful in studying the error probabilities of two kinds. Let $t_{1i}$ and $t_{2i}$ denote respectively the probabilities of type I and type II errors of the $i$-th individual decision in (\ref{INDUCED_DECISION}). Then, by definition, $t_{1i}= \Pr(E(1-\kappa_i|X_i,\tau) > \frac{1}{2}|H_{0i} \mbox{ is true})$ and $t_{2i}= \Pr(E(\kappa_i|X_i,\tau) > \frac{1}{2}|H_{1i} \mbox{ is true})$. It seems that finding the exact asymptotic orders of $t_{1i}$ and $t_{2i}$ is infeasible. Therefore, one convenient and fruitful approach to study their asymptotic behaviors is to find non-trivial asymptotic bounds for them. One way of accomplishing this is to obtain appropriate bounds for either of $E(1-\kappa_i|X_i,\tau)$ and $\Pr(\kappa_i>\eta|X_i,\tau)$ (since $E(\kappa_i|X_i,\tau)$ can be bounded above by $\eta+\Pr(\kappa_i>\eta|X_i,\tau)$, for any $\eta\in(0,1)$), followed by some judicious applications of these bounds.\vspace{1.5mm}

The following theorem is our first step towards this and gives the first concentration inequality involving the posterior distribution of $\kappa_i$'s. Using this theorem, one can derive an upper bound to $E(1-\kappa_i \big|X_i,\tau)$ in a very simple way in case the function $L$ in (\ref{FULL_OG_MODEL}) is bounded above, as indicated in Remark \ref{REM_CONCEN_INEQ_1} below.
 \begin{thm}\label{THM_CONCEN_INEQ_1}
  Suppose $X_i \sim N(\mu_i,1)$ independently for $i=1,\ldots,m$. Consider the one-group prior given in (\ref{FULL_OG_MODEL}) and let $\kappa_i=\frac{1}{1+\lambda_i^2\tau^2}$. Then, for any fixed $\epsilon \in (0,1)$ and any fixed $\tau > 0$,
  \begin{eqnarray}
   \Pr(\kappa_i < \epsilon|X_i,\tau) \leq Ke^{\frac{X_i^2}{2}}\bigg\{ \int_{\frac{1}{\tau^2}(\frac{1}{\epsilon}-1)}^{\infty} t^{-a-1}L(t)dt \bigg\} (1+o(1)), \nonumber %\mbox{ as $\tau \rightarrow 0,$}\nonumber
  \end{eqnarray}
  where the $o(1)$ term above is independent of both the index $i$ and the data point $X_i$, but depends on $\tau$ in such a way that $\lim_{\tau\rightarrow0}o(1)=0$.
   \end{thm}
 
 \begin{cor}\label{COR_CONCEN_INEQ_1}
  Suppose the function $L(\cdot)$ in (\ref{FULL_OG_MODEL}) is uniformly bounded above by some constant $M > 0$. Then under the assumptions of Theorem \ref{THM_CONCEN_INEQ_1}, for any fixed $\epsilon \in (0,1)$ and any fixed $\tau > 0$,
   \begin{eqnarray}
 \Pr(\kappa_i < \epsilon|X_i,\tau) \leq \frac{KM}{a}\epsilon^{a}(1-\epsilon)^{-a}e^{\frac{X_i^2}{2}}\tau^{2a}(1+o(1)),\nonumber
  \end{eqnarray}
  where the $o(1)$ term above is independent of both the index $i$ and the data point $X_i$, but depends on $\tau$ in such a way that $\lim_{\tau\rightarrow0}o(1)=0$.
 \end{cor}

  \begin{remark}\label{REM_CONCEN_INEQ_1}
In case the function $L(\cdot)$ is bounded above by some $M > 0$, then using Corollary \ref{COR_CONCEN_INEQ_1} one can readily obtain the following upper bound on $E(1-\kappa_i \big|X_i,\tau)$:
  \begin{equation}\label{KAPPA_EXP}
   E(1-\kappa_i \big|X_i,\tau)=\int_{0}^{1}\Pr(\kappa_i < \epsilon|X_i,\tau)d\epsilon \leq \frac{KM}{a(1-a)} e^{\frac{X_i^2}{2}} \tau^{2a}(1+o(1)).% \mbox{ as } \tau \rightarrow 0. \nonumber
  \end{equation}
  It has already been shown in Section 2.2 and Section 2.3 that for many of the commonly used shrinkage priors including the horseshoe, the corresponding $L(\cdot)$ is bounded above by some constant $M$. Use of the upper bound from Theorem \ref{THM_CONCEN_INEQ_1} makes the task of finding an upper bound for $E(1-\kappa_i \big|X_i,\tau)$ very simple in such cases. Finding an upper bound for $E(1-\kappa_i \big|X_i,\tau)$ in case of a general $L(\cdot)$, as given in Theorem \ref{THM_MOMENT_INEQ} below, is quite non-trivial and requires pretty delicate arguments based on properties of slowly varying functions.
 \end{remark}
 
 \begin{thm}\label{THM_MOMENT_INEQ}
 Consider the set up of Theorem \ref{THM_CONCEN_INEQ_1}, with the prior on the local shrinkage parameter as in (\ref{FULL_OG_MODEL}) with $a \in (0,1).$ Then, for every fixed $\tau < 1$, 
  \begin{eqnarray}
  E(1-\kappa_i \big|X_i,\tau) \leq \frac{A_{0}K}{a(1-a)}e^{\frac{X_i^2}{2}} \tau^{2a}L(\frac{1}{\tau^2})(1+o(1)),
  \end{eqnarray}
  where the $o(1)$ term above is independent of both the index $i$ and the data point $X_i$, but depends on $\tau$ in such a way that $\lim_{\tau\rightarrow0}o(1)=0$. Here $A_0 \geq 1$ is a constant depending on $L$, such that, $L(\cdot)$ is bounded in every compact subset of $[A_0, \infty)$.
 \end{thm}
 
 The next theorem gives the second concentration inequality of this paper involving the term $\Pr(\kappa_i > \eta|X_i,\tau)$.
 \begin{thm}\label{THM_CONCEN_INEQ_2}
  Under the setup of Theorem \ref{THM_CONCEN_INEQ_1}, for any fixed $\tau > 0$, and each fixed $\eta \in (0,1)$ and $\delta \in (0,1)$,
  \begin{eqnarray}
   \Pr(\kappa_i > \eta|X_i,\tau) &\leq& \frac{H(a,\eta,\delta)e^{-\frac{\eta(1-\delta) X_{i}^{2}}{2}}}{\tau^{2a}\Delta(\tau^2,\eta,\delta)}, \mbox{ uniformly in } X_i \in \mathbb{R},\nonumber\\
   \nonumber\\
%   \end{eqnarray}
%   \begin{eqnarray}
  \mbox{where }  \Delta(\tau^2,\eta,\delta)&=&\xi(\tau^2,\eta,\delta)L\big(\frac{1}{\tau^2}(\frac{1}{\eta\delta}-1)\big),\nonumber\\ \nonumber\\
    \xi(\tau^2,\eta,\delta) &=& \frac{\int_{\frac{1}{\tau^2}\big(\frac{1}{\eta\delta}-1\big)}^{\infty}t^{-(a+\frac{1}{2}+1)}L(t)dt}{(a+\frac{1}{2})^{-1} \big(\frac{1}{\tau^2}\big(\frac{1}{\eta\delta}-1\big)\big)^{-(a+\frac{1}{2})}L(\frac{1}{\tau^2}\big(\frac{1}{\eta\delta}-1\big))},\quad  \mbox{and}\nonumber\\
    \nonumber\\
     H(a,\eta,\delta) &=& \frac{(a+\frac{1}{2}) (1-\eta\delta)^a}{ K(\eta\delta)^{(a+\frac{1}{2})}}.\nonumber
  \end{eqnarray}
   \end{thm}
 \begin{remark}
 It is to be observed in this context that, for $0 < a <1$, the upper bound in Theorem \ref{THM_CONCEN_INEQ_2} of the present article is of a smaller order compared to that derived in Theorem 3.2 of \citet{DG2013}. In particular, using properties of slowly varying functions (see the Appendix), it can be easily established that the ratio of the former to the latter tends to zero as $\tau \rightarrow 0$. The sharper asymptotic bound in Theorem \ref{THM_CONCEN_INEQ_2} results in a sharper asymptotic upper bound to the probability of type-II error, and hence on the overall risk $R_{OG}$ (in Theorem \ref{THM_BAYES_RISK_UBLB}) of the procedure (\ref{INDUCED_DECISION}) as compared to that in \citet{DG2013}. 
 \end{remark}
 Several important features of the posterior distribution of the shrinkage coefficients $\kappa_i$'s based on our general class of tail robust shrinkage priors, now become clear from Theorem \ref{THM_CONCEN_INEQ_1} through Theorem \ref{THM_CONCEN_INEQ_2}. These are listed in Corollary \ref{COR_CONCEN_INEQ_1_2} - Corollary \ref{COR_CONCEN_INEQ_2_2} given below. While Corollary \ref{COR_CONCEN_INEQ_1_2} and Corollary \ref{COR_MOMENT_INEQ_1} are derived using Theorem \ref{THM_CONCEN_INEQ_1} and Theorem \ref{THM_MOMENT_INEQ}, respectively, the rest follow from Theorem \ref{THM_CONCEN_INEQ_2}. Proofs of these results are trivial and hence are omitted. It should however be remembered that these corollaries have no direct use in proving the main theoretical results of this paper.
 
 \begin{cor}\label{COR_CONCEN_INEQ_1_2}
  Under the assumptions of Theorem \ref{THM_CONCEN_INEQ_1}, $\Pr(\kappa_i \geq \epsilon|X_i,\tau) \rightarrow 1$ as $\tau \rightarrow 0$ for any fixed $\epsilon \in (0,1)$ uniformly in $X_i \in \mathbb{R}.$ 
 \end{cor}

 Thus, for each fixed $x \in \mathbb{R}$, the posterior distribution of $\kappa_i$'s, based on the tail robust priors under consideration, tend to concentrate near 1 for small values of $\tau$.
 
 \begin{cor}\label{COR_MOMENT_INEQ_1}
  Under the assumptions of Theorem \ref{THM_CONCEN_INEQ_1}, $E(1-\kappa_i|X_i,\tau) \rightarrow 0$ as $\tau \rightarrow 0$ for any fixed $\epsilon \in (0,1)$ uniformly in $X_i \in \mathbb{R}$. 
 \end{cor}
 Corollary 4.3 above says that for small values of $\tau$, noise observations will be squelched towards the origin by the kind of one-group priors considered in this paper.

%  The next two corollaries can be derived easily using Theorem \ref{THM_CONCEN_INEQ_2}.
 \begin{cor}\label{COR_CONCEN_INEQ_2_1}
   Under the assumptions of Theorem \ref{THM_CONCEN_INEQ_1}, $\Pr(\kappa_i \leq \eta|X_i,\tau) \rightarrow 1 $ as $X_i \rightarrow \infty,$ for any fixed $\tau > 0$ and every fixed $\eta \in (0,1)$.
 \end{cor}
%  \begin{proof}
%  The result follows immediately by using the upper bound obtained in Theorem \ref{THM_CONCEN_INEQ_2} for $\Pr(\kappa_i > \eta|X_i,\tau)$ and letting $X_i$ go to infinity.
%  \end{proof}
%  This shows that for large signals the posterior distribution of the corresponding shrinkage coefficients tend to concentrate near the origin even if the global shrinkage parameter $\tau$ is kept fixed at a small magnitude.
 \begin{cor}\label{COR_CONCEN_INEQ_2_2}
   Under the assumptions of Theorem \ref{THM_CONCEN_INEQ_1}, $E(1-\kappa_i|X_i,\tau) \rightarrow 1 $ as $X_i \rightarrow \infty,$ for any fixed $\tau > 0$.
 \end{cor}
 Corollary 4.5 above shows that, for each of the heavy tailed shrinkage priors under consideration, even if the global variance component $\tau$ is very small, the amount of posterior shrinkage will be negligibly small for large $X_i$'s, thus leaving the large observations almost unshrunk.
 
 \subsection{Asymptotic bounds on probabilities of type I and type II errors when $\tau$ is treated as a tuning parameter}
 
Theorem \ref{THM_T1_UB} and Theorem \ref{THM_T2_UB} below give asymptotic upper bounds to the probability of type I error $(t_{1i})$ and the probability of type II error $(t_{2i})$, respectively, of the $i$-th decision in (\ref{INDUCED_DECISION}), while Theorem \ref{THM_T1_LB} and Theorem \ref{THM_T2_LB} give asymptotic lower bounds for $t_{1i}$ and $t_{2i}$, respectively. As mentioned before, these results lead to the asymptotic bounds on the Bayes risk ($R_{OG}$) of the multiple testing procedure in (\ref{INDUCED_DECISION}).\vspace{1.5mm}

 \begin{thm}\label{THM_T1_UB}
 Suppose $X_1,\cdots,X_m$ are i.i.d. observations having the two-groups mixture distribution described in (\ref{TWO_GROUP_X}) with $\sigma^2=1$ and suppose Assumption \ref{ASSUMPTION_ASYMP} is satisfied by $(\psi^2, p)$. Suppose one is testing $H_{0i}:\nu_i=0$ vs $H_{1i}:\nu_i=1$ using the decision rule (\ref{INDUCED_DECISION}) induced by the general class of one-group shrinkage priors (\ref{FULL_OG_MODEL}) where $a \in (0,1)$ in $\pi(\lambda_i^2)$. Suppose $\tau =\tau_m \rightarrow 0$ as $m \rightarrow \infty$. Then the probability $t_{1i}$ of type I error of the $i$-th decision in (\ref{INDUCED_DECISION}) satisfies 
 \begin{eqnarray}
 t_1\equiv t_{1i}
 \leq \frac{1}{\sqrt{\pi a}} \cdot \frac{2 A_0 K}{a(1-a)} \cdot \frac{\tau^{2a} L(\frac{1}{\tau^2})}{\sqrt{\log(\frac{1}{\tau^2})}}(1+o(1)) \mbox { as } m \rightarrow \infty,\nonumber 
 \end{eqnarray}
where the $o(1)$ term above does not depend on $i$ and tends to zero as $m\rightarrow\infty$. The constant $A_0$ has already been defined in Theorem \ref{THM_MOMENT_INEQ}.
   \end{thm}

\begin{thm}\label{THM_T2_UB}
 Consider the set-up of Theorem \ref{THM_T1_UB} but allow the parameter {\it $a$} to be any positive real number in the definition of the prior $\pi(\lambda_i^2)$ of the local shrinkage parameter in (\ref{FULL_OG_MODEL}). Assume further that $\tau=\tau_m\rightarrow 0$ as $m \rightarrow \infty$ in such a way that $\lim_{m\rightarrow\infty} \frac{\tau}{p}\in(0,\infty)$. Then the probability $t_{2i}$ of type II error of the $i$-th decision in (\ref{INDUCED_DECISION}) satisfies
   \begin{equation}
   t_2 \equiv t_{2i} \leq \big[2\Phi\bigg(\sqrt{\frac{2aC}{\eta(1-\delta)}}\bigg)-1\big]\big(1+o(1)\big) \mbox { as } m \rightarrow \infty,\nonumber 
   \end{equation}
 for every fixed $\eta \in (0,\frac{1}{2})$ and $\delta \in (0,1)$. Here the $o(1)$ term above depends on $\eta \in (0,\frac{1}{2})$ and $\delta \in (0,1)$ and is independent of $i$, and tends to zero as $m \rightarrow \infty$.
   \end{thm}

\begin{remark}\label{REM_T2_UB}
We record here that, instead of $\lim_{m\rightarrow\infty} \tau/p \in(0,\infty)$, if we assume $\log \tau \sim \log p$, keeping the other conditions unaltered, the proof of Theorem \ref{THM_T2_UB} goes through. Consequently, the upper bound on the Bayes risk $R_{OG}$ in Theorem \ref{THM_BAYES_RISK_UBLB}, which has been derived by combining the results of Theorem \ref{THM_T1_UB} and Theorem \ref{THM_T2_UB}, also holds when $\log\tau \sim \log p$, under conditions specified in (I). A similar inspection shows that under conditions specified in (II), the upper bound of Theorem \ref{THM_BAYES_RISK_UBLB} holds when the conditions $\log\tau \sim \log p$ and $\tau=O(p)$ both hold as $m \rightarrow \infty$.
 \end{remark}

 \begin{thm}\label{THM_T1_LB}
  Consider the set-up of Theorem \ref{THM_T1_UB}. Let us fix any $0 < \eta < 1/2$ and any $0< \delta < 1$. Then the probability $t_{1i}$ of type I error of the $i$-th decision in (\ref{INDUCED_DECISION}) satisfies 
  \begin{equation}
   t_1 \equiv t_{1i} \geq \frac{(\frac{1}{2}-\eta)/\sqrt{\pi a}}{H(a,\eta,\delta)}\cdot\frac{\tau^{\frac{2a}{\eta(1-\delta)}}L(\frac{1}{\tau^2})}{\sqrt{\log(\frac{1}{\tau^2})}}(1+o(1))\mbox{ as } m \rightarrow \infty,\nonumber 
  \end{equation}
 where the $o(1)$ term above depends on $\eta \in (0,\frac{1}{2})$ and $\delta \in (0,1)$ and is independent of $i$, and tends to zero as $m \rightarrow \infty$. The constant $H(a,\eta,\delta)$ has already been defined in Theorem \ref{THM_CONCEN_INEQ_2}.
   \end{thm}

  \begin{thm}\label{THM_T2_LB}
  Consider the set-up of Theorem \ref{THM_T2_UB}. Then the probability $t_{2i}$ of type II error of the $i$-th decision in (\ref{INDUCED_DECISION}) satisfies 
    \begin{eqnarray}
    t_2 \equiv t_{2i} \geq (2\Phi(\sqrt{2a}\sqrt{C})-1)(1+o(1)) \mbox { as } m \rightarrow \infty,\nonumber
   \end{eqnarray}
   where the $o(1)$ term above does not depend on $i$ and tends to zero as $m \rightarrow \infty$.
   \end{thm}

 Some important observations regarding an appropriate choice of $\tau$ now follow as consequences of Theorem \ref{THM_T1_UB} - Theorem \ref{THM_T2_LB}. Note that, the type I and type II error probabilities of the $i$-th decision in (\ref{INDUCED_DECISION}), that is, $t_{1i}$ and $t_{2i}$, do not depend on $i$ and their common values are given by $t_1$ and $t_2$, respectively. See the proofs of Theorem \ref{THM_T1_UB} and Theorem \ref{THM_T2_UB} in the Appendix for an explanation of this fact. Thus, the Bayes risk of the decision rules in (\ref{INDUCED_DECISION}) is given by $R_{OG}=mp(\frac{1-p}{p}t_1+t_2)$ (using (\ref{BAYES_RISK_GEN})). Suppose now $\tau\rightarrow 0$ at such a rate that 
  \begin{equation}\label{LB_TEND_TO_INFINITY}
 \frac{(1-p)\tau^{\frac{2a}{\eta(1-\delta)}}L(\frac{1}{\tau^2})}{p\sqrt{\log(\frac{1}{\tau^2})}}\rightarrow \infty \mbox{ as } m \rightarrow \infty.
  \end{equation}
 Then combining Theorem \ref{THM_T1_LB} and Theorem \ref{THM_T2_LB} together, it follows that $\frac{1-p}{p}t_1+t_2 \rightarrow \infty$ as $m\rightarrow\infty$. Consequently, $R_{OG}/R_{Opt}^{BO} \rightarrow \infty$ as $ m \rightarrow \infty$. Consider, for example, the horseshoe or the standard double Pareto prior. For each of these priors, one has $a=0.5$ and the corresponding $L(\cdot)$ has a finite positive limit at infinity as already shown before. Let us take $\tau=p^{\alpha}$, for $\alpha>0$. Now, for $0< \alpha < \frac{1}{2}$, one can always choose some $\eta \in (0,\frac{1}{2})$ and some $\delta \in (0,1)$, such that $0 < \alpha < \eta(1-\delta)<\frac{1}{2}$. As a result, (\ref{LB_TEND_TO_INFINITY}) holds and we have $R_{OG}/R_{Opt}^{BO}\rightarrow \infty \mbox{ as } m \rightarrow \infty$, when $0< \alpha < \frac{1}{2}$. Thus the desired Bayesian optimality property up to $O(1)$ no longer holds in such situations. However, from the derived lower bounds, we can not yet conclude the same if $\alpha\geq \frac{1}{2}$ since, in that case, the quantity in (\ref{LB_TEND_TO_INFINITY}) tends to zero as $m\rightarrow\infty$. But, the upper bound for $\frac{1-p}{p}t_1+t_2$, as obtained by combining Theorem \ref{THM_T1_UB} and Theorem \ref{THM_T2_UB}, tends to infinity as $m\rightarrow\infty$, when $\frac{1}{2}\leq \alpha < 1$. This indicates, though not conclusively, that, for such tail robust priors, $\tau=p^{\alpha}$ with $0<\alpha<1$, is not likely to be a good choice.\vspace{1.5mm}
 
 It should be noted further that the proofs of Theorem \ref{THM_T2_UB} and Theorem \ref{THM_T2_LB} work even if we take $\tau=p^{\alpha}$, for $\alpha \geq 1$ (see Appendix), and the constants in the corresponding asymptotic bounds should be replaced by $\big(2\Phi\big(\sqrt{2a\alpha C/(\eta(1-\delta))}\big)-1\big)$ and $(2\Phi(\sqrt{2a\alpha C})-1)$, respectively. Each of these bounds increases with an increase in $\alpha$ and so does the corresponding bound on the overall Bayes risk $R_{OG}$. Thus, $R_{OG}$ tends to be away from the optimal Bayes risk $R_{Opt}^{BO}$ in (\ref{OPT_BAYES_RISK}), for values of $\alpha >1$, thereby implying that $\tau=p^{\alpha}$ is also not a good choice for $\alpha > 1$. Note that all these arguments remain valid even if we assume $\tau/p^{\alpha}$ has a finite positive limit as $m\rightarrow\infty$. Thus our results gives a partial indication that $\tau=p$ (or, $\lim_{m\rightarrow\infty}\tau/p\in (0,\infty)$) should be the optimal choice of $\tau$. Similar observations were also made by \cite{PKV2014} for the horseshoe prior when $p=o(1)$ as $m\rightarrow \infty$. They found that for optimal contraction of the corresponding posterior distribution around the truth as well as the horseshoe estimator, the optimal choice for $\tau$ would be $\tau=p$ (or, up to some logarithmic factor of it). Moreover, they showed that when $\tau=p^{\alpha}$, the posterior distribution based on the horseshoe prior, contracts around the horseshoe estimator at a sub-optimal rate in the squared $l_2$ sense if $0<\alpha<1$, while it contracts too quickly to yield an adequate measure of uncertainty when $\alpha >1$. Our results are, therefore, in a partial agreement with that of \cite{PKV2014} regarding the optimal choice of $\tau$, when $p$ is assumed to be known.
 
 \subsection{Asymptotic bounds on probabilities of type I and type II errors for the empirical Bayes procedure}
 In Theorem \ref{THM_T1_EB_UB} and Theorem \ref{THM_T2_EB_UB} below we present asymptotic upper bounds to the probabilities of type I and type II errors of the individual decisions corresponding to the empirical Bayes procedure defined in (\ref{INDUCED_DECISION_EB}). Proofs of these theorems are significantly different from those for proving Theorem \ref{THM_T1_UB} and Theorem \ref{THM_T2_UB} when $\tau$ is treated as a tuning parameter. This is so, because for each $i$, $E(1-\kappa_i|X_i,\widehat{\tau})$ depends on the entire data through $\widehat{\tau}$ and $X_i$ in a very complicated manner. So, in order to avoid dealing with the term $E(1-\kappa_i|X_i,\widehat{\tau})$ directly, we need to invoke substantially new arguments as follows. First we divide the range of $\widehat{\tau}$ into two parts. For one part, as observed by \cite{PKV2014}, we use the fact that $E(1-\kappa_i|x,\tau)$ is non-decreasing in $\tau$ for each fixed $x$ and then we use the results of Theorem \ref{THM_T1_UB} and Theorem \ref{THM_T2_UB}, while for the other part, we exploit the structure of the estimator $\widehat{\tau}$ defined in (\ref{TAU_HAT_EMPB}), together with the independence between the $X_i$'s. To the best of our knowledge, arguments of this kind are new in this context and have not been reported elsewhere.
 
 \begin{thm}\label{THM_T1_EB_UB}
 Suppose $X_1,\cdots,X_m$, are i.i.d. observations having the two-groups mixture distribution described in (\ref{TWO_GROUP_X}) with $\sigma^2=1$, and we wish to test the $m$ hypotheses $H_{0i}:\nu_i=0$ vs $H_{1i}:\nu_i=1$, $i=1,\ldots,m$, simultaneously, using the decision rule (\ref{INDUCED_DECISION_EB}) induced by the one-group priors (\ref{FULL_OG_MODEL}) where $a \in (0,1)$ in $\pi(\lambda_i^2)$. Suppose Assumption \ref{ASSUMPTION_ASYMP} is satisfied by $(\psi^2, p)$ with $p \propto m^{-\epsilon}$, for some $0 < \epsilon < 1$. Then, the probability $\widetilde{t}_{1i}$ of type I error of the $i$-th induced decision in (\ref{INDUCED_DECISION_EB}) satisfies 
   \begin{equation}
   \widetilde{t}_{1i} \leq B_{1}^{*}\frac{\alpha_m^{2a}L(\frac{1}{\alpha_m^2})}{\sqrt{\log (\frac{1}{\alpha_m^2})}}(1+o(1))+ \frac{1/\sqrt{\pi}}{m^{c_1/2}\sqrt{\log m}} + e^{-2(2\log 2 -1)\beta mp(1+o(1))}\mbox { as } m \rightarrow \infty,\nonumber
   \end{equation}
 where the $o(1)$ terms appearing above are independent of $i$, and tend to zero as $m \rightarrow \infty$. Here $B^{*}_1$ and $\beta$ are some finite positive constants, each being independent of $m$, while $\alpha_m=\Pr(|X_1|>\sqrt{c_1\log m})$ depends on $m$.
   \end{thm}

 \begin{thm}\label{THM_T2_EB_UB}
 Let us consider the set-up of Theorem \ref{THM_T1_EB_UB}. Then the probability $\widetilde{t}_{2i}$ of type II error of the $i$-th decision in (\ref{INDUCED_DECISION_EB}) satisfies
   \begin{eqnarray}
   \widetilde{t}_{2i} \leq \big[2\Phi\bigg(\sqrt{\frac{2aC}{\eta(1-\delta)}}\bigg)-1\big]\big(1+o(1)\big) \mbox { as } m \rightarrow \infty,\nonumber
   \end{eqnarray}
   for every fixed $\eta \in (0,\frac{1}{2})$ and $\delta \in (0,1)$. Here the $o(1)$ term tends to zero as $m \rightarrow \infty$ and is independent of $i$, but depends on the choices of $\eta \in (0,\frac{1}{2})$ and $\delta \in (0,1)$.
   \end{thm}

 \begin{remark}\label{REM_TAU_CONV}
 Using the architecture of the proofs of Theorem \ref{THM_T1_EB_UB} and Theorem \ref{THM_T2_EB_UB}, one can show now that,
   \begin{equation}\label{CONVERGENCE_TAU}
   \frac{\widehat{\tau}}{p} \stackrel{p}{\longrightarrow} 2\beta \mbox{ as } m\rightarrow \infty,
  \end{equation}
 where the above probability convergence is taken with respect to the joint distribution of $X_i$'s defined through (\ref{TWO_GROUP_X}). Since $\alpha_m\sim 2\beta p$ and $\beta=1-\Phi(c_1C/(2\epsilon))>0$ (see the proof of Theorem \ref{THM_T1_EB_UB}), it will be enough to show that, given any $\delta_0>0$, $\Pr(|\frac{\widehat{\tau}}{\alpha_m}-1|>\delta_0)=o(1)$ as $m\rightarrow \infty$. This follows quite easily using the techniques used for proving Theorem \ref{THM_T1_EB_UB} and Theorem \ref{THM_T2_EB_UB}. Thus, (\ref{CONVERGENCE_TAU}) says that the estimator $\widehat{\tau}$ will be asymptotically of the same order as $p$. Note that when $c_1=2$ and $C/\epsilon$ is small, $2\beta\approx 1$. In that case, $\widehat{\tau}$ is expected to estimate the unknown degree of sparsity $p$ very well.
 \end{remark}
 
\section{Simulations}
In this section, we present and interpret the results obtained in our simulation study. The simulation study has several objectives. The first objective is to motivate the use of the decision rules based on the global-local tail robust priors when data actually come from a two-groups model. Secondly, we want to study empirically the suitability of these priors for handling sparsity as well as their robustness in handling large signals. Thirdly, we want to study the role of $\tau$ as a global shrinkage parameter. The most important objective is to compare the simulation averages of the proportion of misclassified hypotheses (as estimate of the misclassification probability) of these testing rules with that of the Bayes Oracle for the two-groups problem to understand how closely these rules actually perform vis-a-vis the Oracle.\vspace{1.5mm}

We present in this section the numerical results obtained by using the horseshoe prior, the standard double Pareto prior, the Strawderman-Berger prior and the normal-exponential-gamma prior (with $\alpha=1,\beta=0.6$) when the data actually come from a two-groups model. We consider a fully Bayesian approach as well as an empirical Bayes approach. For the fully Bayesian approach, $(\tau, \sigma)$ is assigned a hyperprior given by,
\begin{equation}
   \tau \sim  C^{+}(0,1) \mbox{ and } \pi(\sigma) \mbox{ } \propto \frac{1}{\sigma}.\nonumber 
\end{equation}
and the marginal prior for $\mu_i$'s is obtained by mixing further with respect to this joint prior distribution of $(\tau, \sigma)$. For the empirical Bayes approach, $\sigma$ is taken to be equal to 1, and we use the procedure in (\ref{INDUCED_DECISION_EB}), where we take $c_1=2$ and $c_2=1$ in the definition of $\widehat{\tau}$ in (\ref{TAU_HAT_EMPB}).\vspace{1.5mm}

Our simulation data are generated as follows. For each fixed $p \in (0,1)$, we draw $m=200$ independent observations $X_1, \cdots, X_m$ using the two groups model (\ref{TWO_GROUP_X}), with $\psi_m = \sqrt{2\log{m}}=3.26$ and $\sigma^2 =1$. For estimating the misclassification probability, the process is replicated 1000 times and simulation averages of misclassification proportions are taken as estimates of the misclassification probabilities of the different multiple testing procedures under study. Our results lend support to our theoretical findings and also justification for our theoretical study presented earlier in Section 3 and Section 4.\vspace{1.5mm}
\begin{figure}[ht]
        \begin{center}
                \includegraphics[width=10cm,height=8cm]{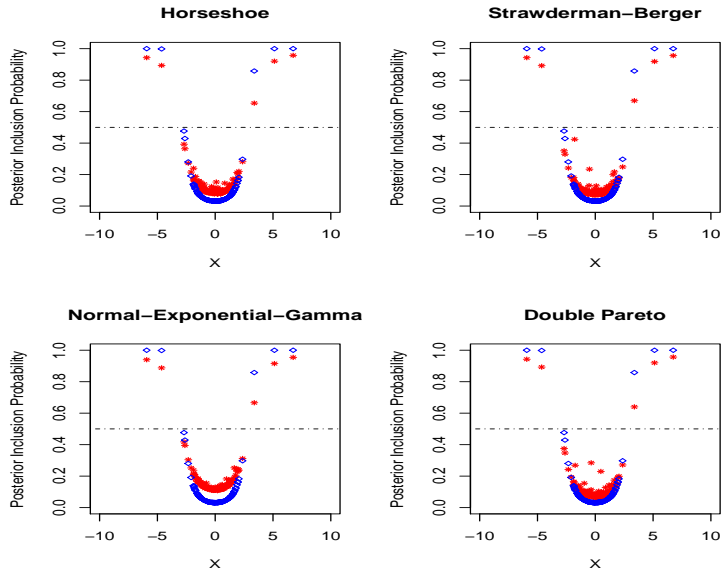}
                \caption{Comparison of the posterior inclusion probabilities and shrinkage coefficients $1-E(\kappa_i|X_1,\cdots,X_m)$ when $p=0.10.$}
        \end{center}
\end{figure}

Taking $p=0.10$, we plot in Figure 1, the theoretical posterior inclusion probabilities $w_i(X_i) = P(\nu_i =1|X_i)$ for the two-groups model (\ref{TWO_GROUP_X}) given by
\begin{equation}
 \omega_i(X_i) = \pi(\nu_i=1|X_i) = \bigg\{\big(\frac{1-p}{p}\big)\sqrt{1+\psi^2}e^{-\frac{X_i^2}{2}\frac{\psi^2}{1+\psi^2}}+1  \bigg\}^{-1},\nonumber
 \end{equation}
along with the shrinkage weights $(1-E(\kappa_i|X_1,\cdots,X_m))$ corresponding to the four one-group shrinkage priors mentioned above against the data. The blue dots in the figure denote the theoretical posterior inclusion probabilities while the red dots correspond to the shrinkage weights $(1-E(\kappa_i|X_1,\cdots,X_m))$. The figures clearly show the proximity of the two quantities for small values of the sparsity parameter $p$ for each of the four shrinkage priors mentioned above. This fact and the theoretical observations made in Section 2 justify the use of $(1-E(\kappa_i|X_1,\cdots,X_m))$ as an approximation to the corresponding posterior inclusion probabilities $\omega_i(X_i)$ in sparse situations and thus motivates the use of decision rules based on $(1-E(\kappa_i|X_1,\cdots,X_m))$ using one-group tail robust shrinkage priors.\vspace{1.5mm}

\begin{figure}[ht]
        \begin{center}
                \includegraphics[width=12cm,height=8cm]{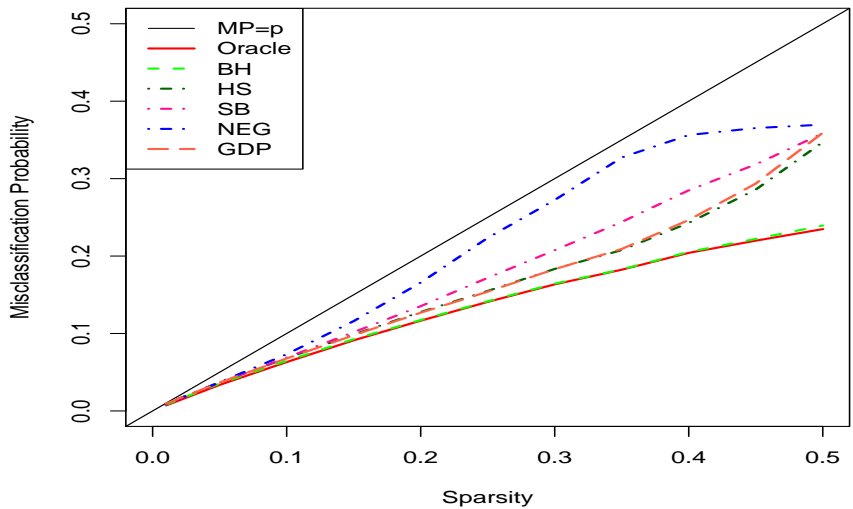}
                \caption{Estimated misclassification probabilities for the full Bayes approach}
        \end{center}
\end{figure}
 
Figure 2 shows the (estimated) misclassification probability (MP) plots of the decision rule (\ref{INDUCED_DECISION_FB}) corresponding to the four priors under consideration along with those of the Bayes Oracle and the Benjamini-Hochberg rule against values of $p\in \{0.01, 0.05, 0.1, 0.15, 0.2, 0.25, 0.3, 0.35, 0.4, 0.45, 0.5 \}.$ As mentioned before, these are obtained as average values of misclassification proportions of the decision rule (\ref{INDUCED_DECISION_FB}) and that corresponding to the Bayes Oracle (defined in Section 3.1). The Bayes Oracle serves as the lower bound to the MP whereas the line $MP=p$ corresponds to the situation when we reject all null hypotheses without looking into the data. It is clear from Figure 2 that when the sparsity parameter $p$ is small, the MP plots corresponding to the four priors under consideration almost coincide with that of the Bayes Oracle which is in conformity with the theoretical results of the present article. While the MP plots of the horseshoe and the standard double Pareto prior are nearly identical, the MP plot corresponding to the the Strawderman-Berger prior is in close proximity. The same is true for the MP plot for the normal-exponential-gamma plot, for small values of $p$. When $p$ is larger, say above 0.4, performance of each of these priors become inferior compared to the Bayes Oracle. We have also plotted the MP for the Benjamini-Hochberg rule, for $\alpha = 1/\log{m}=0.1887$. \cite{BCFG2011} theoretically established that for such choices of $\alpha$, the corresponding Benjamini-Hochberg rule becomes ABOS in the present set up. This is also corroborated by Figure 2 where we see that the Benjamini-Hochberg rule achieves practically the same MP as the Bayes Oracle. Similar phenomenon can also be observed for the empirical Bayes procedure (\ref{INDUCED_DECISION_EB}) also in Figure 3 below. It should be noted that, a comparison between Figure 2 and Figure 3 clearly shows that performance of the normal-exponential-gamma prior (with $\alpha=1$ and $\beta=0.6$) in terms of the overall MP in the empirical Bayes approach is substantially better compared to its full Bayes counterpart.\vspace{1.5mm} 	

\begin{figure}[ht]
        \begin{center}
                \includegraphics[width=12cm,height=8cm]{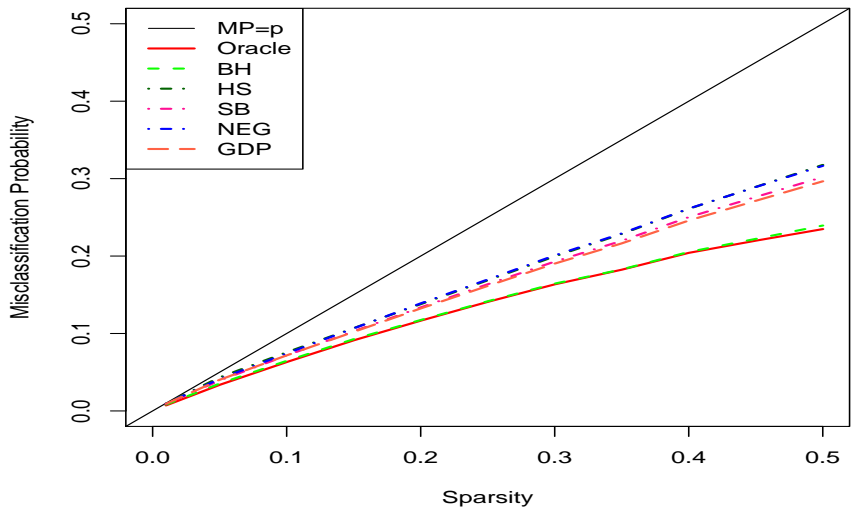}
                \caption{Estimated misclassification probabilities for the empirical Bayes approach}
        \end{center}
\end{figure}

\begin{figure}[ht]
        \begin{center}
                \includegraphics[width=12cm,height=8cm]{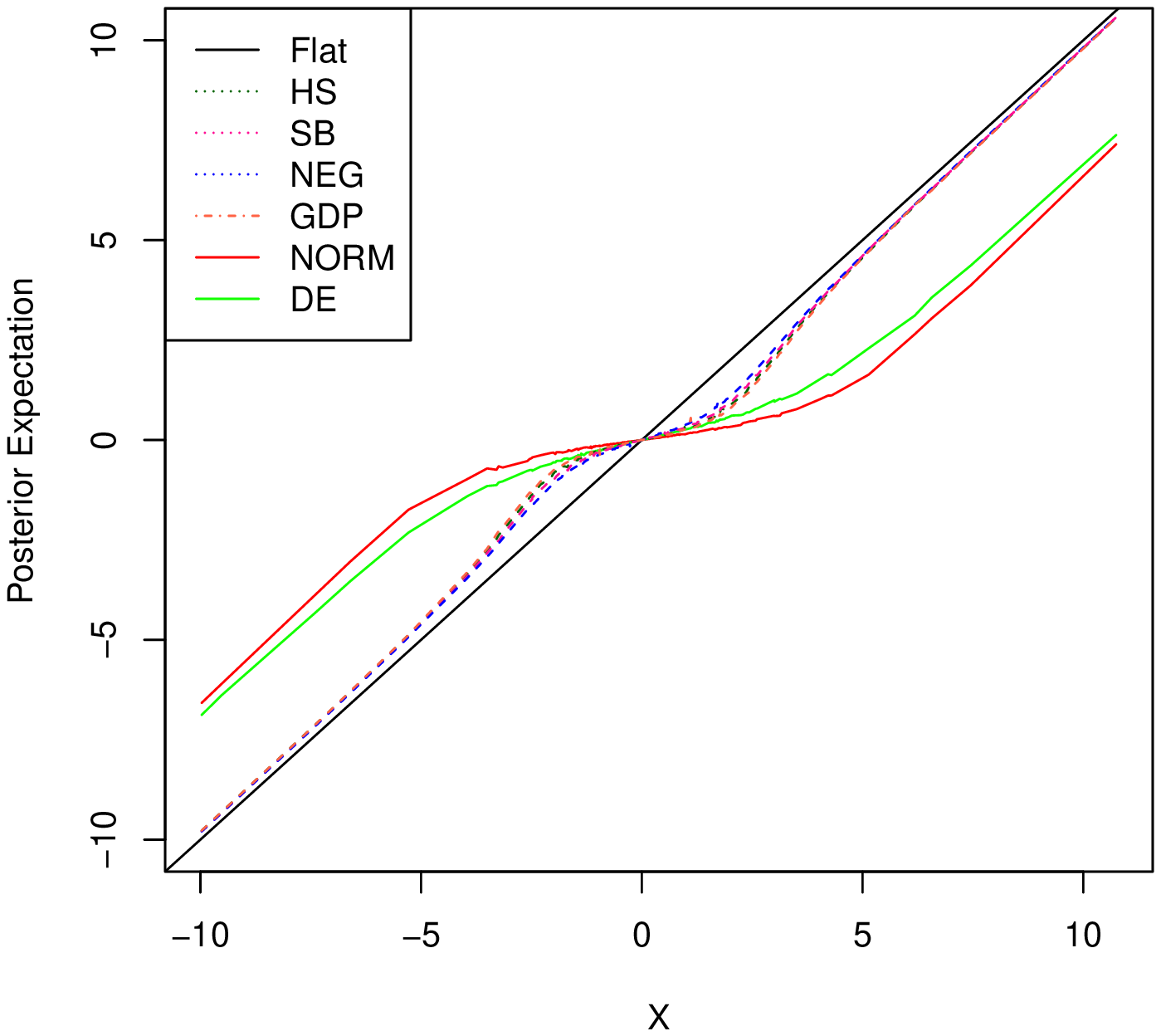}
                \caption{Posterior Mean $E(\mu|X)$ versus $X$ plot for p=0.25}
        \end{center}
\end{figure}
Shrinkage properties corresponding to the four tail-robust priors under consideration along with the Laplace and the half-normal priors having exponential tails are demonstrated through Figure 4. Here we plot the posterior expectations $E(\mu_i|X_1,\cdots,X_m)$ against different values ($X_i$) of the observations. Figure 4 clearly shows that the noise observations are shrunk towards zero efficiently while the big signals are left mostly unshrunk by each of the four tail robust shrinkage priors under consideration, while the normal and the Laplace shrink even the large signals by some non-diminishing amounts. Similar observations was also made in \cite{DG2013} about the Laplace and the half-normal priors.\vspace{1.5mm}
\begin{figure}[ht]
        \begin{center}
                \includegraphics[width=12cm,height=8cm]{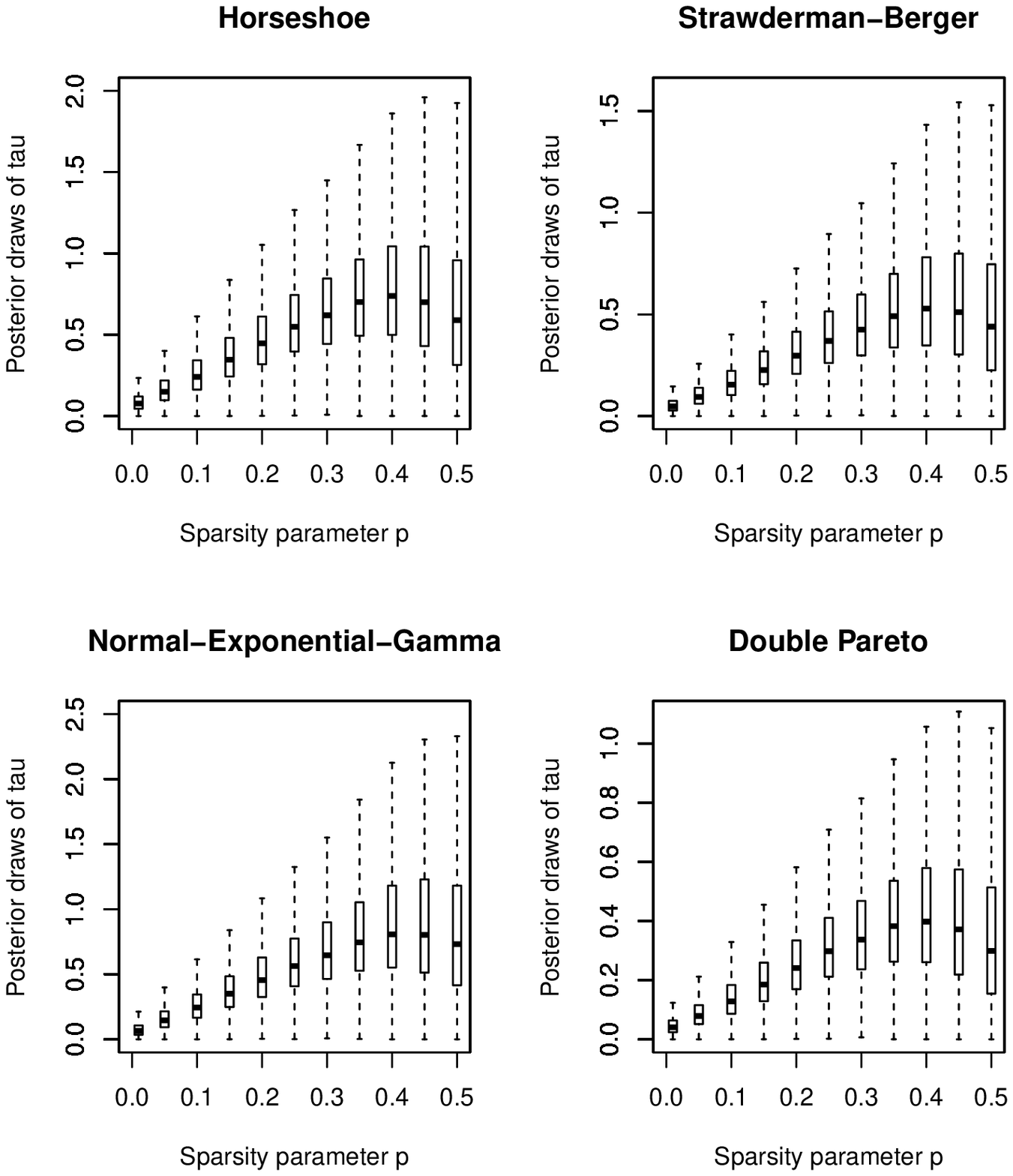}
                \caption{Posterior draws of tau for the horseshoe, the Strawderman-Berger, the normal-exponential-gamma and the standard double Pareto priors across different levels of sparsity}
        \end{center}
\end{figure}

Finally, we demonstrate in Figure 5 how the global shrinkage parameter $\tau$ adapts to the sparsity level of the data in the full Bayes approach. We draw box-plots of the posterior draws for $\tau$ across different levels of sparsity for each of the four priors in focus. It may be observed that box-plots are highly concentrated near zero for small values of $p$ while as $p$ gets large, the range of the box-plots systematically become wider and the median of the posterior samples of $\tau$ show an overall trend. Thus the overall sparsity level of the data is reflected well through the posterior distribution of the global shrinkage parameter $\tau$. Similar phenomenon was observed in \cite{DG2013} in the context of the horseshoe prior. This indicates how global shared parameters can control error rates in multiple testing by estimating the overall sparsity level, as already discussed in \cite{SB2006} and \cite{CPS2010}.

\section{Discussion}
 We have considered in this paper multiple hypothesis testing under sparsity in a decision theoretic framework. Global-local shrinkage priors are used towards this end. We have proved an Oracle property of the resulting decision rules similar to those of \cite{DG2013}. We have also considered an empirical Bayes version of the induced decisions by using an estimate of the global shrinkage parameter $\tau$ and shown its Oracle optimality property under very mild restriction over the sparsity parameter $p$. One of the salient features of our work is that we have provided unified results for a very general class of shrinkage priors including some of the commonly used priors such as the horseshoe prior, generalized double Pareto priors, normal-exponential-gamma priors and many others. As a special case of our general result, we have strengthened the optimality result for the horseshoe prior as considered in \cite{DG2013}. Further, we have settled a conjecture of these authors related to generalized double Pareto priors. Our technique of proof shows that some of the arguments of \cite{DG2013} can be simplified. Moreover, in the theoretical treatment of this paper, we have exploited properties of slowly varying functions to obtain a general unifying argument that works across a large class of one-group priors for investigating their theoretical properties, which to the best of our knowledge, has not been done before and can be very useful in other contexts.\vspace{1.5mm}

 The Oracle optimality property of the multiple testing rules studied in this paper, assumes $\sigma^2$ to be known and treats $\tau$ as a tuning parameter or it is estimated from the data. A natural question is whether these decision rules retain this optimality property if we use a full Bayes approach by assigning a hyperprior to $(\tau,\sigma)$. Thus one would like to know whether the decision rules in (\ref{INDUCED_DECISION_FB}) enjoy similar Bayesian optimality property. Our simulation results indicate that this is indeed the case, though giving a formal theoretical justification for the same is difficult. To elucidate this point, let us first take a look at the posterior means of $\mu_i$'s based on which the full Bayes decision rules in (\ref{INDUCED_DECISION_FB}) are defined. For each $i=1,\dots,m$, it is given by,
 \begin{eqnarray}
  E(\mu_i|X_1,\cdots,X_m)
  &=&\int_{0}^{\infty}\int_{0}^{\infty}E(\mu_i|X_1,\cdots,X_m,\tau,\sigma)\pi(\tau,\sigma|X_1,\cdots,X_m)d\tau d\sigma\nonumber\\
  &=& (1-E(\kappa_i|X_1,\cdots,X_m))X_i\nonumber
 \end{eqnarray}
 where $E(\kappa_i|X_1,\cdots,X_m)=E_{(\tau,\sigma|X_1,\cdots,X_m)}E(\kappa_i|X_1,\cdots,X_m,\tau,\sigma)$. Here, the posterior distribution of $\kappa_i$ depends on the entire dataset and not only on $X_i$, which is the case when $\tau$ is treated as a tuning parameter and $\sigma^2$ is assumed to be known. Also, for each $i$, the posterior distribution $\pi(\kappa_i|X_1,\cdots,X_m)$ is analytically quite intractable in this case. Therefore, finding estimates or asymptotic bounds on the type I and type II error measures directly using this posterior distribution does not look feasible. By looking at the expression for $E(\kappa_i|X_1,\ldots,X_m)$ given above, one may however wish to explore for possibilities of using the concentration and moment inequalities based on $E(\kappa_i|X_1,\ldots,X_m,\tau,\sigma)$ first (as given in Section 4.2), and then finding suitable bounds for $E(\kappa_i|X_1,\ldots,X_m)$ using $\pi(\tau,\sigma|X_1,\ldots,X_m)$. But that also seems very non-trivial as we do not have any proper handle over $\pi(\tau,\sigma|X_1,\ldots,X_m)$. The problem does not become any easier even if we assume $\sigma^2$ to be known and put a hyperprior for $\tau$ only. It seems that one needs to invoke significantly new techniques and arguments for a full Bayes treatment of the type of one-group priors studied in this paper. We leave this as an interesting problem for future research.\vspace{1.5mm}
     
We expect that Oracle optimality properties like those studied in this paper, should now easily be obtained even if one considers $\pi(\lambda_i^2) \sim (\lambda_i^2)^{-a-1}L(\lambda_i^2)$ as $\lambda_i^2 \rightarrow \infty$, as originally considered in Theorem 1 of \cite{PS2011}, by employing the techniques given in this paper together with the fact that the ratio of these two functions belongs to a small neighborhood of 1 not containing the origin, for all sufficiently large $\lambda_i^2$. It may be commented at this point that we have provided a set of sufficient conditions that are quite general in nature under which such Bayesian optimality results hold true. However, characterizing priors of this kind having such optimality property is quite challenging and remains an interesting and open problem till date.\vspace{1.5mm}
  
Our results bear the potential of application also in variable selection. In particular, we want to examine how to extend the proposed decision rule for selection of regression parameters and study optimality of such decision rules in that context. It has been mentioned earlier that \cite{PKV2014} considered the problem of estimation of a sparse multivariate normal vector using the horseshoe estimator. They showed among other things that the horseshoe estimator achieves (up to a multiplicative factor) the minimax squared error risk and the corresponding posterior distribution contracts around the truth at the minimax rate. A similar optimality result was proved for an empirical Bayes version of the horseshoe estimator. The thing to be noted is that the ``optimal'' choice of $\tau$ that makes all the good results come through is given by the theoretical proportion of non-zero entries $p$ in the mean vector (up to a logarithmic factor), and choices of $\tau$ of the order of $p^{\alpha}$ for $\alpha > 1$ and $\alpha < 1$ were shown to be sub-optimal for such purposes. It is interesting to note, as already argued in Section 4.2, that a choice of $\tau$ which is asymptotically of the order of $p$ seems to be the optimal choice of $\tau$ in our case, when $p$ is assumed to be known. Therefore choice of $\tau$ in the vicinity of $p$ seems to be optimal for the two most important inferential problems. In a recent technical report, \cite{GC2015} adopted the framework of \cite{PKV2014} and proved generalizations of their results using Bayes estimators coming from a very general class of one-group tail robust priors and taking $\tau$ proportional to the fraction $p$. Proof of these results crucially exploit the inequalities given in Theorem \ref{THM_MOMENT_INEQ} and Theorem \ref{THM_CONCEN_INEQ_2} of the present article. It is also worth noting that the empirical Bayes estimate of $\tau$ that turned out to be useful in the paper of \cite{PKV2014} also turns out to be equally handy in the present multiple testing problem. We have already commented that a possible reason for this is the fact that the empirical Bayes estimate $\widehat{\tau}$ proposed by \cite{PKV2014}, is asymptotically of the same order of $p$ under the present two-groups formulation.\vspace{1.5mm}
 
Like \cite{DG2013} or \cite{PKV2014}, our theoretical results mostly treat the global shrinkage parameter $\tau$ as a tuning parameter. As explained before, a fully Bayesian approach using a hyperprior on $\tau$ seems difficult to handle. The difficulty in analyzing the posterior distribution for a full Bayes analysis corresponding to horseshoe-type priors with heavy tails has also been addressed in \cite{BPPD2012, BPPD2014} and \cite{PBPD2014}. The authors of these articles studied the contraction properties of global-local shrinkage priors and proved their sub-optimal properties in cases where the tails of the distributions decay at an exponential or faster rate. However, they conjectured that heavy tailed priors, such as, the horseshoe, the generalized double Pareto and the normal-exponential-gamma priors, should have optimal posterior contraction properties. This has already been established in \cite{PKV2014} for the horseshoe prior and subsequently by \cite{GC2015} for a broader class of one-group tail-robust priors which is rich enough to include, among others, the horseshoe, the generalized double Pareto and the normal-exponential-gamma priors in particular. Results of the last two articles, however, assume $\tau$ to be a tuning parameter. On the other hand, \cite{BPPD2012, BPPD2014} and \cite{PBPD2014} used full Bayes approaches for studying concentration probabilities around sparse vectors and rate of convergence of the corresponding posterior distributions based on their proposed class of shrinkage priors, referred to as Dirichlet-Laplace (DL)-type priors. Such priors result in a prior distribution for the individual means which is highly peaked near the origin and has thick tails. However, the DL-type priors are fundamentally different from the global-local scale mixture representation of \cite{PS2011} and hence cannot be covered by our general class of tail robust one-group priors. For example, \cite{BPPD2014} used a set of scalars $(\phi_1\tau,\ldots,\phi_m\tau)$ instead of a single global shrinkage component $\tau$ and assigned a {\it Dirichlet} prior to $(\phi_1,\ldots,\phi_m)$, with a common Dirichlet concentration parameter. Moreover, the asymptotics of contraction for the DL-type priors are quite different from those derived in this paper, and hence do not seem to apply to our present asymptotic framework, at least, as far as we can see. However, we would like to emphasis that studying the posterior contraction properties of the kind of one-group shrinkage priors considered in this paper, using a full Bayes approach, remains an interesting and open research problem till date. We hope to address this problem somewhere else in future.

% Recently, there has been a growing interest in identifying the ``rare and weak'' signals and also in estimating the proportion of true signals using a two-groups formulation slightly different from the one considered in this paper. See, for example, \cite{ING1997}, \cite{DJ2004}, \cite{MR2006}, \cite{CJL2007} and \cite{CJ2010}, in this context. An interesting and related problem to consider is when both the mean and the variance of the signal density vary with the number of tests $m$ under such two-groups framework. For example, people have considered the case when the mean tends to infinity at a slower rate, but the variance remains bounded as $m$ goes to infinity. Moreover, in order to look at the feasibility of any one-group formulation, one needs to first carefully understand what kind of shrinkage priors, if any, would possibly be a good approximation to the two-groups formulation in such contexts. We hope to address these issues elsewhere in future.
 
\appendix  
\section{Appendix}
\label{app}
\subsection{Some Important Properties of Slowly Varying Functions}
The theoretical results derived in this paper depend heavily up on some fundamental properties of general slowly varying functions. These are listed as Lemma \ref{LEM_A_1} - \ref{LEM_A_4}, given below. Interested readers are referred to the classic text of \cite{BGT1987} for a detailed treatment of these results.

 \begin{lem}\label{LEM_A_1}
  If $L$ is any slowly varying function and $\alpha < -1,$ then
  \begin{equation}
   \frac{\int_{x}^{\infty} t^{\alpha}L(t)dt}{x^{\alpha+1}L(x)} \sim \frac{-1} {\alpha+1} \mbox{ as } x \rightarrow \infty.\nonumber
  \end{equation}
 \end{lem}
%  \begin{proof}
%   See Proposition 1.5.10 of \cite{BGT1987}.
%  \end{proof}

 \begin{lem}\label{LEM_A_2}
 If $L$ is any slowly varying function then there exists $A_0 > 0$ such that $L$ is locally bounded in $[A_0,\infty)$, that is, $L$ is bounded in all compact subsets of $[A_0,\infty).$
 \end{lem}
%  \begin{proof}
%   See Lemma 1.3.2 and subsequent discussion in \cite{BGT1987}.
%  \end{proof}
 
  \begin{lem}\label{LEM_A_3}
  If $L$ is any slowly varying function, $A_0$ is so large such that $L$ is locally bounded in $[A_0,\infty)$ and $\alpha > -1,$ then
  \begin{equation}
   \frac{\int_{A_0}^{x} t^{\alpha}L(t)dt}{x^{\alpha+1}L(x)} \sim \frac{1}{1+\alpha} \mbox{ as } x \rightarrow \infty.\nonumber
  \end{equation}
 \end{lem}
%  \begin{proof}
%   See Proposition 1.5.8 of \cite{BGT1987}.
%  \end{proof}

 \begin{lem}\label{LEM_A_4}
  If $L$ is any slowly varying function then
  \begin{enumerate}[(i)]
   \item $\lim\limits_{x \rightarrow \infty} \frac{\log{L(x)}}{\log{(x)}}=0,$
   \item $L^{\beta}$ is slowly varying for all $\beta \in \mathbb{R},$ and,
   \item $\lim\limits_{x \rightarrow \infty} x^{-\alpha}L(x)=0,$ for all $\alpha > 0$.
  \end{enumerate}
 \end{lem}
%  \begin{proof}
%   See Proposition 1.3.6 of \cite{BGT1987}.
%   \end{proof}  

 \subsection{Proofs}
 This section contains the proofs of all the major theoretical results of this paper. But before proving the main theoretical results given in Section 3, namely, Theorem \ref{THM_BAYES_RISK_UBLB} and Theorem \ref{THM_BAYES_RISK_EB_UB}, we shall first present the proofs of the results of Section 4 since these results are essential for deriving Theorem \ref{THM_BAYES_RISK_UBLB} and Theorem \ref{THM_BAYES_RISK_EB_UB}. But above all, let us first prove the following lemma, which we shall use often for proving rest of the results of this paper.
 \begin{lem}\label{LEM_A_5}
 Let $L : (0,\infty) \rightarrow (0, \infty)$ be a measurable function and $a$ be a real number such that $\int_{0}^{\infty}t^{-a-1} L(t) dt = K^{-1} $, $K \in (0, \infty)$. Then, assuming $\tau \rightarrow 0$,
  \begin{equation}
  \int_{0}^{1} u^{a+\frac{1}{2}-1}(1-u)^{-a-1}L\bigg(\frac{1}{\tau^2}\big(\frac{1}{u}-1\big)\bigg)du
   = K^{-1}\tau^{-2a}(1+o(1)),\nonumber 
  \end{equation}
 where the $o(1)$ term above depends on $\tau$ in such a way that $\lim_{\tau \rightarrow 0} o(1)=0$. 
 \end{lem}
\begin{proof}
% \textbf{{Proof:}}
 Let $J = \int_{0}^{1} u^{a+\frac{1}{2}-1}(1-u)^{-a-1}L\big(\frac{1}{\tau^2}\big(\frac{1}{u}-1\big)\big)du.$ Since the integrand in $J$ is non-negative, using the change of variable $t=\frac{1}{\tau^2}\big(\frac{1}{u}-1\big)$, one has
 \begin{equation}
  J = (\tau^2)^{-a}\int_{0}^{\infty}(1+t\tau^2)^{-\frac{1}{2}} t^{-a-1}L(t)dt.\nonumber\\
 \end{equation}
 Since $t^{-a-1}L(t)$ is assumed to be integrable, the proof follows immediately using Lebesgue's Dominated Convergence Theorem, where the $o(1)$ term is such that $\lim\limits_{\tau \rightarrow 0}o(1)=0$.
 \end{proof}
\textbf{{Proof of Theorem \ref{THM_CONCEN_INEQ_1}:}}
Fix any $\epsilon \in (0,1).$ Then by definition,
\begin{eqnarray}
  \Pr(\kappa_i < \epsilon |X_i,\tau)
   &=& \frac{\int_{0}^{\epsilon} \kappa_i^{a+\frac{1}{2}-1}(1-\kappa_i)^{-a-1}L\big(\frac{1}{\tau^2}\big(\frac{1}{\kappa_i}-1\big)\big) e^{-\frac{\kappa_i X_{i}^{2}}{2}}d\kappa_i}{\int_{0}^{1}\kappa_i^{a+\frac{1}{2}-1}(1-\kappa_i)^{-a-1}L\big(\frac{1}{\tau^2}\big(\frac{1}{\kappa_i}-1\big)\big) e^{-\frac{\kappa_i X_{i}^{2}}{2}}d\kappa_i}\nonumber\\
   &\leq& \frac{e^{\frac{X_i^2}{2}}\int_{0}^{\epsilon} \kappa_i^{a+\frac{1}{2}-1}(1-\kappa_i)^{-a-1}L\big(\frac{1}{\tau^2}\big(\frac{1}{\kappa_i}-1\big)\big)d\kappa_i}{\int_{0}^{1}\kappa_i^{a+\frac{1}{2}-1}(1-\kappa_i)^{-a-1}L\big(\frac{1}{\tau^2}\big(\frac{1}{\kappa_i}-1\big)\big) d\kappa_i}\nonumber  
 \end{eqnarray}
 which follows from the fact that $e^{-\frac{X_i^2}{2}} \leq e^{-\frac{\kappa_i X_i^2}{2}} \leq 1$ for every $\kappa_i \in (0,1)$. Now using the change of variable $t=\frac{1}{\tau^2}\big(\frac{1}{\kappa_i}-1\big)$ to the numerator of the right hand side of the above inequality and applying Lemma \ref{LEM_A_5} to the corresponding denominator, we obtain,
 \begin{equation}
 \Pr(\kappa_i < \epsilon |X_i,\tau)
%   &\leq& \frac{e^{\frac{X_i^2}{2}} (\tau^2)^{-a}\int_{\frac{1}{\tau^2}(\frac{1}{\epsilon}-1)}^{\infty} (1+t\tau^2)^{-\frac{1}{2}} t^{-a-1} L(t)dt}{K^{-1}\tau^{-2a}(1+o(1))}, \nonumber\\
  \leq Ke^{\frac{X_i^2}{2}}\bigg\{ \int_{\frac{1}{\tau^2}(\frac{1}{\epsilon}-1)}^{\infty} t^{-a-1}L(t)dt \bigg\} (1+o(1)).\nonumber
 \end{equation}
Here the $o(1)$ term is independent of $X_i$ and tends to zero as $\tau\rightarrow 0$. This completes the proof of Theorem \ref{THM_CONCEN_INEQ_1}.
\vspace{1.5mm}

\textbf{{Proof of Theorem \ref{THM_MOMENT_INEQ}:}}
  Since $L$ is slowly varying and $a \in (0,1)$, by Lemma \ref{LEM_A_3}, there exists some $A_0 > 0$ such that $L$ is locally bounded on $[A_0, \infty)$ and
  \begin{equation}\label{THM_4.2_EQ1}
   \lim_{z \rightarrow \infty}\frac{\int_{A_{0}}^{z} t^{-a}L(t)dt}{z^{1-a}L(z)} = \frac{1}{1-a}. 
  \end{equation}
  Without any loss of generality, one may assume that $A_0 \geq 1$. Now observe that from the definition of $E(1-\kappa_i \big|X_i,\tau)$, it directly follows that,
   \begin{eqnarray}\label{MOMENT_INEQ_1}
  E(1-\kappa_i \big|X_i,\tau)
   &=& \frac{\int_{0}^{1} \kappa_i^{a+\frac{1}{2}-1}(1-\kappa_i)^{-a}L\big(\frac{1}{\tau^2}\big(\frac{1}{\kappa_i}-1\big)\big) e^{-\frac{\kappa_i X_{i}^{2}}{2}}d\kappa_i}{\int_{0}^{1}\kappa_i^{a+\frac{1}{2}-1}(1-\kappa_i)^{-a-1}L\big(\frac{1}{\tau^2}\big(\frac{1}{\kappa_i}-1\big)\big) e^{-\frac{\kappa_i X_{i}^{2}}{2}}d\kappa_i}\nonumber\\
   &\leq& \frac{e^{\frac{X_i^2}{2}}\int_{0}^{1} \kappa_i^{a+\frac{1}{2}-1}(1-\kappa_i)^{-a}L\big(\frac{1}{\tau^2}\big(\frac{1}{\kappa_i}-1\big)\big)d\kappa_i}{\int_{0}^{1}\kappa_i^{a+\frac{1}{2}-1}(1-\kappa_i)^{-a-1}L\big(\frac{1}{\tau^2}\big(\frac{1}{\kappa_i}-1\big)\big) d\kappa_i}
 \end{eqnarray}
  
 Then, using the change of variable $t=\frac{1}{\tau^2}\big(\frac{1}{\kappa_i}-1\big)$ to the numerator of (\ref{MOMENT_INEQ_1}) and applying Lemma \ref{LEM_A_5} to its denominator, we obtain the following:
  \begin{eqnarray}
  E(1-\kappa_i \big|X_i,\tau)
%    &=& \frac{\int_{0}^{1} \kappa_i^{a+\frac{1}{2}-1}(1-\kappa_i)^{-a}L\big(\frac{1}{\tau^2}\big(\frac{1}{\kappa_i}-1\big)\big) e^{-\frac{\kappa_i X_{i}^{2}}{2}}d\kappa_i}{\int_{0}^{1}\kappa_i^{a+\frac{1}{2}-1}(1-\kappa_i)^{-a-1}L\big(\frac{1}{\tau^2}\big(\frac{1}{\kappa_i}-1\big)\big) e^{-\frac{\kappa_i X_{i}^{2}}{2}}d\kappa_i}\nonumber\\
%    &\leq& \frac{\int_{0}^{1} \kappa_i^{a+\frac{1}{2}-1}(1-\kappa_i)^{-a}L\big(\frac{1}{\tau^2}\big(\frac{1}{\kappa_i}-1\big)\big)d\kappa_i}{e^{-\frac{X_i^2}{2}}\int_{0}^{1}\kappa_i^{a+\frac{1}{2}-1}(1-\kappa_i)^{-a-1}L\big(\frac{1}{\tau^2}\big(\frac{1}{\kappa_i}-1\big)\big) d\kappa_i}\nonumber\\
  &\leq& \frac{e^{\frac{X_i^2}{2}} (\tau^2)^{-a+1}\int_{0}^{\infty} (1+t\tau^2)^{-\frac{3}{2}} t^{-a} L(t)dt}{K^{-1}\tau^{-2a}(1+o(1))}\nonumber\\
  &=& Ke^{\frac{X_i^2}{2}} \tau^{2}\int_{0}^{\infty} (1+t\tau^2)^{-\frac{3}{2}} t^{-a} L(t)dt(1+o(1))\nonumber\\
  &=& KJ_{\tau}e^{\frac{X_i^2}{2}}(1+o(1)), \mbox{ say,} \label{THM_4.2_EQ2}
 \end{eqnarray}
 where $J_{\tau} = \tau^2\int_{0}^{\infty} (1+t\tau^2)^{-\frac{3}{2}} t^{-a} L(t)dt$ and the $o(1)$ term in (\ref{THM_4.2_EQ2}) does not depend on $X_i$ nor on the index $i$, and tends to zero as $\tau\rightarrow 0$.\vspace{1.5mm}
 
Now we observe that for any $\tau < 1$, we can split $J_{\tau}$ as
  %  $J_{\tau}=\int_{0}^{\infty} \frac{t\tau^2}{(1+t\tau^2)^{\frac{3}{2}}} t^{-a-1}L(t)dt.$ Then, for every $\tau^2 < \frac{1}{A_0},$ one can bound the term $J_{\tau}$ as follows:
  \begin{eqnarray}
   J_{\tau}= \bigg(\int_{0}^{A_{0}} + \int_{A_{0}}^{\frac{A_{0}}{\tau^2}} + \int_{\frac{A_{0}}{\tau^2}}^{\infty}\bigg) \frac{t\tau^2}{(1+t\tau^2)^{\frac{3}{2}}} t^{-a-1}L(t)dt = J_{1\tau}+ J_{2\tau} + J_{3\tau}, \mbox{ say} \nonumber.
   \end{eqnarray}
  First note that 
   \begin{equation} \label{J1_EQ}
   J_{1\tau} \leq \int_{0}^{A_{0}} \frac{A_{0}\tau^2}{(1+t\tau^2)^{\frac{3}{2}}} t^{-a-1}L(t)dt \leq A_o \tau^2 K^{-1}.
   \end{equation}
  Next, we have
   \begin{eqnarray}
   J_{2 \tau} 
%    \leq \tau^2\int_{A_{0}}^{\frac{A_{0}}{\tau^2}} t^{-a}L(t)dt
  \leq  \frac{A_{0}^{1-a}\tau^{2a}}{1-a}L\big(\frac{A_{0}}{\tau^2}\big)(1+o(1))
   \leq \frac{A_{0}\tau^{2a}}{1-a}L(\frac{1}{\tau^2})(1+o(1)). \label{J2_EQ}
    \end{eqnarray}
  where the above inequality in (\ref{J2_EQ}) comes from using (\ref{THM_4.2_EQ1}) and then the slowly varying property of $L$, together with the fact that $A_0 \geq 1$ and $a \in (0,1)$.\vspace{1.5mm} 
  
  Finally, using Lemma \ref{LEM_A_1} and the slowly varying property of $L$, we have
    \begin{equation} \label{J3_EQ}
    J_{3\tau} \leq \frac{A_{0}^{-a}\tau^{2a}}{a}L\big(\frac{A_{0}}{\tau^2}\big)(1+o(1)) \leq \frac{A_{0}\tau^{2a}}{a}L(\frac{1}{\tau^2})(1+o(1)).
    \end{equation}
  Using (\ref{J1_EQ}), (\ref{J2_EQ}) and (\ref{J3_EQ}) it follows that
    \begin{eqnarray}
    J_{\tau} &\leq& \frac{A_{0}\tau^{2a}L(\frac{1}{\tau^2})}{a(1-a)}\bigg[K^{-1}\frac{(\frac{1}{\tau^2})^{a-1}}{L(\frac{1}{\tau^2})} + a(1+o(1))+(1-a)(1+o(1))\bigg] \nonumber\\
   &=& \frac{A_{0}\tau^{2a}L(\frac{1}{\tau^2})}{a(1-a)}(1+o(1)). \label{THM_4.2_EQ3}
  \end{eqnarray}  
 The equality in (\ref{THM_4.2_EQ3}) follows as $\lim_{\tau \rightarrow 0} (\frac{1}{\tau^2})^{a-1}/L(\frac{1}{\tau^2})=0$, which, in turn, follows from Proposition part (iii) of Lemma \ref{LEM_A_4}. Theorem \ref{THM_MOMENT_INEQ} now follows from (\ref{THM_4.2_EQ2}) and (\ref{THM_4.2_EQ3}).\vspace{1.5mm}
 
\textbf{{Proof of Theorem \ref{THM_CONCEN_INEQ_2}:}}
  Fix $\eta \in (0,1)$ and $\delta \in (0,1).$ Then by definition,
  \begin{eqnarray}
     \Pr(\kappa_i > \eta |X_i,\tau)
     &=& \frac{\int_{\eta}^{1} \kappa_i^{a+\frac{1}{2}-1}(1-\kappa_i)^{-a-1} L\big(\frac{1}{\tau^2}\big(\frac{1}{\kappa_i}-1\big)\big) e^{-\frac{\kappa_i X_{i}^{2}}{2}}d\kappa_i}{\int_{0}^{1} \kappa_i^{a+\frac{1}{2}-1}(1-\kappa_i)^{-a-1} L\big(\frac{1}{\tau^2}\big(\frac{1}{\kappa_i}-1\big)\big) e^{-\frac{\kappa_i X_{i}^{2}}{2}}d\kappa_i}\nonumber\\
    &\leq& \frac{e^{-\frac{\eta(1-\delta) X_{i}^{2}}{2}} \int_{\eta}^{1} \kappa_i^{a+\frac{1}{2}-1}(1-\kappa_i)^{-a-1} L\big(\frac{1}{\tau^2}\big(\frac{1}{\kappa_i}-1\big)\big)d\kappa_i}{\int_{0}^{\eta\delta} \kappa_i^{a+\frac{1}{2}-1}(1-\kappa_i)^{-a-1} L\big(\frac{1}{\tau^2}\big(\frac{1}{\kappa_i}-1\big)\big)d\kappa_i}.\nonumber
%     \label{THM3.3_INEQ1}
  \end{eqnarray}
  Now applying the change of variable $t=\frac{1}{\tau^2}\big(\frac{1}{\kappa_i}-1\big)$ to both the numerator and denominator on the right hand side of the preceding inequality, we obtain:
  \begin{eqnarray}
%     \frac{\Pr(\kappa_i > \eta |X_i,\tau)}{\Pr(\kappa_i \leq \eta |X_i,\tau)}
   \Pr(\kappa_i > \eta |X_i,\tau)
%     &=& \frac{\int_{\eta}^{1} \kappa_i^{a+\frac{1}{2}-1}(1-\kappa_i)^{-a-1} L\big(\frac{1}{\tau^2}\big(\frac{1}{\kappa_i}-1\big)\big) e^{-\frac{\kappa_i X_{i}^{2}}{2}}d\kappa_i}{\int_{0}^{1} \kappa_i^{a+\frac{1}{2}-1}(1-\kappa_i)^{-a-1} L\big(\frac{1}{\tau^2}\big(\frac{1}{\kappa_i}-1\big)\big) e^{-\frac{\kappa_i X_{i}^{2}}{2}}d\kappa_i}\nonumber\\
%      &\leq& \frac{\exp\big(-\frac{\eta X_{i}^{2}}{2}\big) \int_{\eta}^{1} \kappa_i^{a+\frac{1}{2}-1}(1-\kappa_i)^{-a-1} L\big(\frac{1}{\tau^2}\big(\frac{1}{\kappa_i}-1\big)\big)d\kappa_i}{\int_{0}^{\eta\delta} \kappa_i^{a+\frac{1}{2}-1}(1-\kappa_i)^{-a-1} L\big(\frac{1}{\tau^2}\big(\frac{1}{\kappa_i}-1\big)\big) e^{-\frac{\kappa_i X_{i}^{2}}{2}}d\kappa_i}\nonumber\\
%     &\leq& \frac{e^{-\frac{\eta(1-\delta) X_{i}^{2}}{2}} \int_{\eta}^{1} \kappa_i^{a+\frac{1}{2}-1}(1-\kappa_i)^{-a-1} L\big(\frac{1}{\tau^2}\big(\frac{1}{\kappa_i}-1\big)\big)d\kappa_i}{\int_{0}^{\eta\delta} \kappa_i^{a+\frac{1}{2}-1}(1-\kappa_i)^{-a-1} L\big(\frac{1}{\tau^2}\big(\frac{1}{\kappa_i}-1\big)\big)d\kappa_i}\nonumber\\
     &\leq& \frac{e^{-\frac{\eta(1-\delta) X_{i}^{2}}{2}} (\tau^2)^{-a} \int_{0}^{\frac{1}{\tau^2}\big(\frac{1}{\eta}-1\big)} (1+t\tau^2)^{-\frac{1}{2}}t^{-a-1}L(t)dt}{(\tau^2)^{-a} \int_{\frac{1}{\tau^2}\big(\frac{1}{\eta\delta}-1\big)}^{\infty} (1+t\tau^2)^{-\frac{1}{2}}t^{-a-1}L(t)dt}\nonumber\\
%     &\leq& \frac{e^{-\frac{\eta(1-\delta) X_{i}^{2}}{2}}\int_{0}^{\infty} t^{-a-1}L(t)dt}{(\tau^2)^{-\frac{1}{2}}\int_{\frac{1}{\tau^2}\big(\frac{1}{\eta\delta}-1\big)}^{\infty} \sqrt{\frac{t\tau^2}{1+t\tau^2}}t^{-(a+\frac{1}{2}+1)}L(t)dt}\nonumber\\
    &\leq& \frac{K^{-1}e^{-\frac{\eta(1-\delta) X_{i}^{2}}{2}}}{\sqrt{\frac{1-\eta\delta}{\tau^2}} \int_{\frac{1}{\tau^2}\big(\frac{1}{\eta\delta}-1\big)}^{\infty}t^{-(a+\frac{1}{2}+1)}L(t)dt }\nonumber\\
%     &=& \frac{K^{-1}e^{-\frac{\eta(1-\delta) X_{i}^{2}}{2}}}{\sqrt{\frac{1-\eta\delta}{\tau^2}}\xi(\tau^2,\eta,\delta)(a+\frac{1}{2})^{-1}\big(\frac{1}{\tau^2}\big(\frac{1}{\eta\delta}-1\big)\big)^{-(a+\frac{1}{2})}L(\frac{1}{\tau^2}\big(\frac{1}{\eta\delta}-1\big))}\nonumber\\
%     &\leq& \frac{K^{-1}e^{-\frac{\eta(1-\delta) X_{i}^{2}}{2}}}{\sqrt{\eta\delta}\xi(\tau^2,\eta,\delta)a^{-1}\big(\frac{1}{\tau^2}\big(\frac{1}{\eta\delta}-1\big)\big)^{-a}L(\frac{1}{\tau^2}\big(\frac{1}{\eta\delta}-1\big))}\nonumber\\
%     &=& \frac{aK^{-1}(1-\eta\delta)^{a}(\eta\delta)^{-(a+\frac{1}{2})} e^{-\frac{\eta(1-\delta) X_{i}^{2}}{2}}}{\tau^{2a}\xi(\tau^2,\eta,\delta)L(\frac{1}{\tau^2}\big(\frac{1}{\eta\delta}-1\big))}\nonumber\\
    &=& \frac{H(a,\eta,\delta)e^{-\frac{\eta(1-\delta) X_{i}^{2}}{2}}}{\tau^{2a}\Delta(\tau^2,\eta,\delta)},\nonumber
   \end{eqnarray}
  where we use the fact that $\int_{0}^{\frac{1}{\tau^2}(\frac{1}{\eta}-1)}\frac{t^{-a-1}L(t)}{\sqrt{1+t\tau^2}}dt \leq \int_{0}^{\infty} t^{-a-1}L(t)dt =K^{-1}$ and $\frac{t\tau^2}{1+t\tau^2} \geq \frac{1-\eta\delta}{\tau^2}$ for each $t \geq \frac{1}{\tau^2}\big(\frac{1}{\eta\delta}-1\big)$ as $t \mapsto t\tau^2/(1+t\tau^2)$ is an increasing function of $t$ for any fixed $\tau^2 > 0$. This completes the proof of Theorem \ref{THM_CONCEN_INEQ_2}.\vspace{1.5mm}

\textbf{{Proof of Theorem \ref{THM_T1_UB}:}}
  First note that the form of the posterior density of $\kappa_i$ given $(X_i, \tau^2)$ as a function of $(X_i, \tau^2)$ is the same for each $i=1,\ldots,m$. Also under $H_{0i}$, the distribution of $X_{i}$ does not depend on $i$. Therefore the probability~$t_{1i}= \Pr(E(1-\kappa_i|X_i,\tau) > \frac{1}{2}|H_{0i} \mbox{ is true})$ of type I error for the $i$-th test is the same for each $i$, and we denote it by $t_1$. Using Theorem \ref{THM_MOMENT_INEQ}, for any $\tau <1$, the event $\left \{ E(1-\kappa_i|X_i,\tau) > \frac{1}{2} \right \}$ implies the following event
  \begin{align}
 \left \{ X_i^2 > 2a \log (\frac{1}{\tau^2}) - 2 \log L(\frac{1}{\tau^2}) -2 \log(\frac{2 A_0 K}{a(1-a)}) -2 \log (1+o(1)) \right \},\nonumber  
  \end{align}
 where the $o(1)$ term tends to zero as $\tau \rightarrow 0$ and is independent of $X_i$. Note that $2a \log (\frac{1}{\tau^2}) \rightarrow \infty$ as $\tau \rightarrow 0$, $\log(1+o(1)) \rightarrow 0$ as $\tau \rightarrow 0$ and $2A_0K/(a(1-a))$ is a positive constant. Also, using part (i) of Lemma \ref{LEM_A_4}, $\lim_{\tau \rightarrow 0} \log(L(\frac{1}{\tau^2}))/\log(\frac{1}{\tau^2})=0.$ Therefore, we have, for all sufficiently small $\tau < 1$, 
     \begin{equation}\label{T1_BOUND}
  t_{1} \leq \Pr\bigg(X_i^2 > 2\bigg\{a\log\big(\frac{1}{\tau^{2}}\big)+\log(\frac{1}{L(\frac{1}{\tau^2})})+\log(\frac{a(1-a)}{2A_0 K})\bigg\}\bigg|H_{0i} \mbox{ is true}\bigg) (1+o(1)),
     \end{equation}
  where $\lim_{\tau \rightarrow 0} o(1) = 0$ and this term does not depend on $i$, since under $H_{0i}$, the distribution of $X_i$ is $N(0,1)$ independently of $i$. Also note that for all sufficiently small $\tau < 1$,
  \begin{equation}\label{T1_BOUND_OBS_1}
  \bigg\{a\log\big(\frac{1}{\tau^{2}}\big)+\log(\frac{1}{L(\frac{1}{\tau^2})})+\log(\frac{a(1-a)}{2A_0 K})\bigg\} > 0. 
  \end{equation}
 Using (\ref{T1_BOUND}) and (\ref{T1_BOUND_OBS_1}), and the fact that $1-\Phi(t) < \frac{\phi(t)}{t}$, for $t > 0$ we have, as $\tau \rightarrow 0$,   
  \begin{eqnarray}
       t_1 &\leq& \Pr\bigg(|Z| > \sqrt{2a\log\big(\frac{1}{\tau^{2}}\big)+2\log(\frac{1}{L(\frac{1}{\tau^2})})+2\log(\frac{a(1-a)}{2A_0K})} \bigg) (1+o(1)) \nonumber \\ 
       &\leq& 2\cdot \frac{\phi\big(\sqrt{2a\log\big(\frac{1}{\tau^{2}}\big)+2\log(\frac{1}{L(\frac{1}{\tau^2})})+2\log(\frac{a(1-a)}{2A_0K})}\big)}{\sqrt{\big(2a\log\big(\frac{1}{\tau^{2}}\big)+2\log(\frac{1}{L(\frac{1}{\tau^2})})+2\log(\frac{a(1-a)}{2A_0K})\big)}}(1+o(1)) \nonumber \\ 
       &=& \frac{1}{\sqrt{\pi a}} \cdot \frac{2 A_0 K}{a(1-a)} \cdot \frac{\tau^{2a} L(\frac{1}{\tau^2})}{\sqrt{\log(\frac{1}{\tau^2})}}(1+o(1)).\nonumber 
  \end{eqnarray}
 In the above the $o(1)$ term tends to zero as $\tau \rightarrow 0$, and clearly it does not depend on $i$. The last equality follows using the functional form of $\phi(\cdot)$ in the numerator and by getting an asymptotic expression for the denominator using the observations made earlier about the relative magnitudes of the different terms as $\tau \rightarrow 0$. Since as $m \rightarrow \infty$, $\tau =\tau_m \rightarrow 0$, all the limiting statements used in the theorem also hold when $m \rightarrow \infty$.\vspace{1.5mm}

\textbf{{Proof of Theorem \ref{THM_T2_UB}:}}    
  Let us fix any $\eta \in (0,\frac{1}{2})$ and any $\delta \in (0,1).$ Using the inequality
  \begin{equation}
   \kappa_i \leq 1\{\eta < \kappa_i \leq 1\} + \eta,\nonumber
  \end{equation}
  we get the following:
  \begin{eqnarray}\label{THM4.5_UB1}
   E(\kappa_i|X_i,\tau) \leq \Pr(\kappa_i > \eta|X_i,\tau)+ \eta.
  \end{eqnarray}
  Equation (\ref{THM4.5_UB1}) coupled with Theorem \ref{THM_CONCEN_INEQ_2}, implies that for every $X_i \in \mathbb{R}$ we have,
  \begin{eqnarray}
   \bigg\{E(\kappa_i|X_i,\tau) > \frac{1}{2}\bigg\}
%    &\subseteq& \bigg\{\Pr(\kappa_i > \eta|X_i,\tau) > \frac{1}{2}-\eta\bigg\}\nonumber\\
   &\subseteq& \bigg\{\frac{H(a,\eta,\delta)e^{-\frac{\eta(1-\delta) X_{i}^{2}}{2}}}{\tau^{2a}\Delta(\tau^2,\eta,\delta)}> \frac{1}{2}-\eta\bigg\}.\label{THM4.5_UB2}
  \end{eqnarray}
  Using exactly similar argument as used in Theorem \ref{THM_T1_UB} about equality of $t_{1i}$ for $i=1,\ldots,m$, but now noting that under $H_{1i}$, the distribution of $X_i$ does not depend on $i$, it follows that $t_{2i}=t_2$ for some $t_2$ for $i=1,\ldots,m$, where $t_{2i}$ denotes the probability of type II error of the $i$-th test. Now observe that for sufficiently large $m$, $\tau = \tau_m < 1$ and hence $\log (\frac{1}{\tau^2}) \neq 0$. Therefore, using (\ref{THM4.5_UB2}), we have, 
  \begin{eqnarray}
   t_2 &=& \Pr\big(E(\kappa_i|X_i,\tau) > \frac{1}{2}\big |H_{1i} \mbox{ is true}\big)\nonumber\\
       &\leq& \Pr\bigg(\frac{H(a,\eta,\delta)e^{-\frac{\eta(1-\delta) X_{i}^{2}}{2}}}{\tau^{2a}\Delta(\tau^2,\eta,\delta)} > \frac{1}{2}-\eta \big |H_{1i} \mbox{ is true} \bigg)\nonumber\\
       &=& \Pr\bigg(X_i^2 < \frac{2}{\eta(1-\delta)}\bigg(a\log(\frac{1}{\tau^2})+\log(\frac{1}{\Delta(\tau^2,\eta,\delta)})+\log\big(\frac{H(a,\eta,\delta)}{\frac{1}{2}-\eta}\big) \bigg) |H_{1i} \mbox{ is true}\bigg)\nonumber \\
       &=& \Pr\bigg(X_i^2 < \frac{2a}{\eta(1-\delta)} \log(\frac{1}{\tau^2}) (1+o(1)) \big |H_{1i} \mbox{ is true}\bigg) \mbox{ as }m \rightarrow \infty.\label{T2_EQ_LAST}
             \end{eqnarray}
 To prove the equality in (\ref{T2_EQ_LAST}), we first observe that $\log (\frac{1}{\tau^2}) \rightarrow \infty$ as $m \rightarrow \infty$ and that $\log\big(\frac{H(a,\eta,\delta)}{\frac{1}{2}-\eta}\big)$ is a bounded quantity. Furthermore, one can show that
 \begin{equation}\label{THM4.5_UB3_JUSTIFICATION}
   \frac{\log(\frac{1}{\Delta(\tau^2,\eta,\delta)})}{\log(\frac{1}{\tau^2})} \rightarrow 0 \mbox{ as } \tau \rightarrow 0 \mbox{ and hence as }m \rightarrow \infty.  
  \end{equation}
 Combination of these facts prove the last equality in (\ref{T2_EQ_LAST}). For the time being let us assume (\ref{THM4.5_UB3_JUSTIFICATION}), and the proof of this will be given at the end.\vspace{1.5mm}

 To complete the proof of the theorem let us proceed as follows. Note that under $H_{1i}$, $X_i \sim N(0, 1+\psi^2)$. Therefore, by (\ref{T2_EQ_LAST}) and the fact that $\lim\limits_{m \rightarrow \infty} \frac{\psi^2}{1+\psi^2} = 1$ under Assumption \ref{ASSUMPTION_ASYMP}, we have 
\begin{equation} \label{T2_EQUATION}
  t_2 \leq  \Pr\bigg(|Z| < \sqrt{\frac{2a}{\eta(1-\delta)}}\sqrt{\frac{\log(\frac{1}{\tau^2})}{\psi^2}} (1+o(1))\bigg) \mbox{ as }m \rightarrow \infty.
\end{equation} 
Now under the assumption that $\lim_{m\rightarrow\infty} \tau/p \in(0,\infty)$, it follows that under Assumption \ref{ASSUMPTION_ASYMP}, $\log(\frac{1}{\tau^2})/\psi^2\rightarrow C \in (0,\infty)$ as $m \rightarrow \infty.$ Together with (\ref{T2_EQUATION}), this shows
\begin{eqnarray}
 t_2 &\leq& \Pr \bigg( |Z| < \sqrt{\frac{2aC}{\eta(1-\delta)}} \big(1+o(1)\big) \bigg) \mbox{ as }m \rightarrow \infty\nonumber\\
  &=& \Pr \bigg(|Z| < \sqrt{\frac{2aC}{\eta(1-\delta)}}\bigg)\big(1+o(1)\big)\mbox{ as }m \rightarrow \infty \nonumber \\
  &=& \bigg[2\Phi\bigg(\sqrt{\frac{2aC}{\eta(1-\delta)}}\bigg)-1\bigg]\big(1+o(1)\big) \mbox{ as }m \rightarrow \infty. \nonumber
\end{eqnarray}
It is clear from the proof that the $o(1)$ terms do not depend on $i$. This completes the proof of the theorem, modulo the proof of (\ref{THM4.5_UB3_JUSTIFICATION}) which is given below.\vspace{1.5mm}    

Since $L(\cdot)$ is slowly varying and $a >0,$ from Lemma \ref{LEM_A_1} it follows that, for every fixed $\eta \in (0,1)$ and $\delta \in (0,1),$ 
%  we obtain:
%    \begin{eqnarray}
%    \int_{\frac{1}{\tau^2}\big(\frac{1}{\eta\delta}-1\big)}^{\infty} t^{-(a+\frac{1}{2}+1)}L(t)dt = (a+\frac{1}{2})^{-1}\big(\frac{1}{\tau^2}\big(\frac{1}{\eta\delta}-1\big)\big)^{-(a+\frac{1}{2})}L(\frac{1}{\tau^2}\big(\frac{1}{\eta\delta}-1\big))(1+o(1))\mbox{ as $\tau \rightarrow 0$},\nonumber
%   \end{eqnarray}
%   whence 
  \begin{eqnarray}\label{THM4.5_UB4}
   \lim_{\tau \rightarrow 0} \xi(\tau^2,\eta,\delta) = \lim_{\tau \rightarrow 0}\frac{\int_{\frac{1}{\tau^2}\big(\frac{1}{\eta\delta}-1\big)}^{\infty}t^{-(a+\frac{1}{2}+1)}L(t)dt}{(a+\frac{1}{2})^{-1}\big(\frac{1}{\tau^2}\big(\frac{1}{\eta\delta}-1\big)\big)^{-(a+\frac{1}{2})}L(\frac{1}{\tau^2}\big(\frac{1}{\eta\delta}-1\big))}= 1,
  \end{eqnarray}
  where $\xi(\tau^2,\eta,\delta)$ is defined in the statement of Theorem \ref{THM_CONCEN_INEQ_2}. Again, using Lemma \ref{LEM_A_3}, we obtain
  \begin{equation}\label{THM4.5_UB5}
   \lim_{\tau \rightarrow 0} \frac{\log{L(\frac{1}{\tau^2}\big(\frac{1}{\eta\delta}-1\big))}}{\log{\frac{1}{\tau^2}}}=0,
  \end{equation}
  for every fixed $\eta \in (0,1)$ and every fixed $\delta \in (0,1),$ since $L(\cdot)$ is slowly varying. (\ref{THM4.5_UB4}) and (\ref{THM4.5_UB5}), together with the definition of $\Delta(\tau^2,\eta,\delta)$, lead to (\ref{THM4.5_UB3_JUSTIFICATION}) immediately.\vspace{1.5mm}

\textbf{{Proof of Theorem \ref{THM_T1_LB}:}}
 By definition, the probability of type I error for the $i$-th decision in (\ref{INDUCED_DECISION}) is given by,
  \begin{equation}
   t_1 = \Pr(E(1-\kappa_i|X_i,\tau) > \frac{1}{2}|H_{0i} \mbox{ is true})\nonumber
  \end{equation}
where $t_1$ does not depend on $i$ as already shown in the proof of Theorem \ref{THM_T1_UB} before.\vspace{1.5mm}

Now, using (\ref{THM4.5_UB1}) and Theorem (\ref{THM_CONCEN_INEQ_2}), for every fixed $\eta \in (0,1/2)$ and every fixed $\delta \in (0,1)$, we have,
 \begin{equation}
   E(\kappa_i|X_i,\tau) 
%   \leq \eta + \Pr(\kappa_i > \eta|X_i,\tau)
  \leq \eta +  \frac{H(a,\eta,\delta)e^{-\frac{\eta(1-\delta) X_{i}^{2}}{2}}}{\tau^{2a}\Delta(\tau^2,\eta,\delta)}\nonumber
 \end{equation}
 whence it follows that
 \begin{eqnarray}\label{T1_LB_EQ1}
  \bigg\{E(1-\kappa_i|X_i,\tau) > \frac{1}{2} \bigg\}
%  &\equiv&\bigg\{E(\kappa_i|X_i,\tau) < \frac{1}{2} \bigg\}\nonumber\\
 &\supseteq& \bigg\{\frac{H(a,\eta,\delta)e^{-\frac{\eta(1-\delta) X_{i}^{2}}{2}}}{\tau^{2a}\Delta(\tau^2,\eta,\delta)} < \frac{1}{2}-\eta \bigg\}.
 \end{eqnarray}

 Therefore, using the definition of $t_1$ and (\ref{T1_LB_EQ1}), we obtain the following:
%  for every fixed $\eta \in (0,1/2)$ and every fixed $\delta \in (0,1)$, the following:
 \begin{eqnarray}
  t_1 
%   &=& \Pr(E(1-\kappa_i|X_i,\tau) > \frac{1}{2}|H_{0i} \mbox{ is true})\nonumber\\
  &\geq& \Pr\bigg(\frac{H(a,\eta,\delta)e^{-\frac{\eta(1-\delta) X_{i}^{2}}{2}}}{\tau^{2a}\Delta(\tau^2,\eta,\delta)} < \frac{1}{2}-\eta|H_{0i} \mbox{ is true}\bigg)\nonumber\\
  &=& \Pr\bigg(X^2_i >\frac{2}{\eta(1-\delta)}\{a\log(\frac{1}{\tau^2})+\log(\frac{1}{\Delta(\tau^2,\eta,\delta)})+ \log(\frac{H(a,\eta,\delta)}{\frac{1}{2}-\eta})\}|H_{0i} \mbox{ is true}\bigg).\nonumber
 \end{eqnarray}
 Then using (\ref{THM4.5_UB3_JUSTIFICATION}) coupled with the arguments as in the proof of Theorem \ref{THM_T1_UB}, we obtain,
 \begin{equation}
  t_1 \geq \frac{(\frac{1}{2}-\eta)/\sqrt{\pi a}}{H(a,\eta,\delta)}\cdot\frac{\tau^{\frac{2a}{\eta(1-\delta)}}}{\sqrt{\log(\frac{1}{\tau^2})}}\Delta(\tau^2,\eta,\delta)(1+o(1))\mbox{ as }m \rightarrow \infty,\nonumber 
 \end{equation}
 where the $o(1)$ term above does not depend on $i$, and tends to zero as $m\rightarrow\infty$.\vspace{1.5mm}
 
 Since $\Delta(\tau^2,\eta,\delta) \sim L(\frac{1}{\tau^2})$ as $\tau \rightarrow 0$ and hence as $m \rightarrow \infty$, the stated result follows immediately.
%  we have,
%   \begin{equation}
%   t_1 \geq \frac{(\frac{1}{2}-\eta)/\sqrt{\pi a}}{H(a,\eta,\delta)}\cdot\frac{\tau^{\frac{2a}{\eta(1-\delta)}}L(\frac{1}{\tau^2})}{\sqrt{\log(\frac{1}{\tau^2})}}(1+o(1)) \mbox{ as } m \rightarrow \infty.\nonumber 
%  \end{equation}
This completes the proof of Theorem \ref{THM_T1_LB}.\vspace{1.5mm}
% \end{proof}

\textbf{{Proof of Theorem \ref{THM_T2_LB}:}}
  By definition, the probability of type II error for the $i$-th decision in (\ref{INDUCED_DECISION}) is given by,
  \begin{equation}
   t_2 = \Pr\big(E(1-\kappa_i|X_i,\tau) \leq \frac{1}{2}\big |H_{1i} \mbox{ is true}\big)\nonumber
  \end{equation}
  where $t_2$ does not depend on $i$ as already shown in the proof of Theorem \ref{THM_T2_UB} before.\vspace{1.5mm}
  
  Note that, by our assumption, $\tau\rightarrow 0$ as $m\rightarrow\infty$. Therefore, using Theorem \ref{THM_MOMENT_INEQ}, for all sufficiently large $m$, we have,
  \begin{align}
   \bigg\{\frac{A_{0}K}{a(1-a)}e^{\frac{X_i^2}{2}} \tau^{2a}L(\frac{1}{\tau^2})(1+o(1)) \leq \frac{1}{2} \bigg\} \subseteq \bigg\{E(1-\kappa_i|X_i,\tau) \leq \frac{1}{2}\bigg\},\nonumber
  \end{align}
  where the $o(1)$ term tends to zero as $m\rightarrow\infty$, and does not depend on $i$ or $X_i$. Hence, for all such $m$,
   \begin{eqnarray}
   t_2 &=& \Pr\big(E(1-\kappa_i|X_i,\tau) \leq \frac{1}{2}\big |H_{1i} \mbox{ is true}\big)\nonumber\\
   &\geq&  \Pr\big(\frac{A_{0}K}{a(1-a)}e^{\frac{X_i^2}{2}} \tau^{2a}L(\frac{1}{\tau^2})(1+o(1)) \leq \frac{1}{2}\big |H_{1i} \mbox{ is true}\big)\nonumber\\
   &\geq&  \Pr\big(X^2_i \leq 2\big\{\log(\frac{1}{\tau^{2a}})+\log(\frac{1}{L(\frac{1}{\tau^2})})+\log(\frac{a(1-a)}{2A_{0}K})+\log(1+o(1))\} \big|H_{1i} \mbox{ is true}\big).\nonumber
   \end{eqnarray}
 Since $\lim\limits_{\tau \rightarrow 0} {\log(\frac{1}{L(\frac{1}{\tau^2})})}/{\log(\frac{1}{\tau^{2}})}=0$ and $\log(\frac{1}{\tau^{2}})\rightarrow \infty$ as $\tau \rightarrow 0$, we have, for all sufficiently large $m$,
%    Now using the facts that $\lim_{\tau \rightarrow 0} \frac{\log(\frac{1}{L(\frac{1}{\tau^2})})}{\log(\frac{1}{\tau^{2}})}=0$ and $\log(\frac{1}{\tau^{2}})\rightarrow \infty$ as $\tau \rightarrow 0$, we have,
   \begin{align}
    2\big\{\log(\frac{1}{\tau^{2a}})+\log(\frac{1}{L(\frac{1}{\tau^2})})+\log(\frac{a(1-a)}{2A_{0}K})+\log(1+o(1))\big\} &=2a\log(\frac{1}{\tau^{2}})(1+o(1)) > 0.\nonumber
   \end{align}
 
 Since $X_i \sim N(0, 1+\psi^2)$ under $H_{1i}$ and $\psi^2 \rightarrow \infty$, we have,
%  and $\psi^2 \rightarrow \infty$, the probability of type II error of the induced decisions (\ref{INDUCED_DECISION}), satisfies,
   \begin{eqnarray}
    t_2 &\geq& \Pr\big(X^2_i \leq 2a\log(\frac{1}{\tau^{2}})(1+o(1))\big|H_{1i} \mbox{ is true}\big)\nonumber\\
        &=& \Pr\big(|Z| \leq \sqrt{2a}\sqrt{\frac{\log(\frac{1}{\tau^{2}})}{\psi^2}}(1+o(1))\big)\nonumber\\
%         &=& \Pr\big(|Z| \leq \sqrt{2a}\sqrt{C})(1+o(1))\nonumber\\
        &=& 2(\Phi(\sqrt{2a}\sqrt{C})-1)(1+o(1)).\nonumber
   \end{eqnarray}

The second inequality in the above chain of inequalities follows using Assumption \ref{ASSUMPTION_ASYMP} and the fact that $\lim\limits_{m\rightarrow\infty} \tau/p \in(0,\infty)$. This completes the proof of Theorem \ref{THM_T2_LB}.\vspace{1.5mm}
 
Next we move on to the proofs of Theorem \ref{THM_T1_EB_UB} and Theorem \ref{THM_T2_EB_UB}. But, before that, we would like to mention that all the $o(1)$ terms used in the proofs of these two theorems, are independent of the index $i$ and tend to zero as the number of tests $m$ goes to infinity.\vspace{1.5mm}
 
\textbf{{Proof of Theorem \ref{THM_T1_EB_UB}:}}
We fix any $c_1 \geq 2$ and $c_2 \geq 1$ in the definition of $\widehat{\tau}$ in (\ref{TAU_HAT_EMPB}). \vspace{1.5mm}

% Now, using the facts that, when $\nu_1=0$, $X_1|\nu_1=0 \sim N(0,1)$ and given $\nu_1=1$, $X_1|\nu_1=1 \sim N(0,1+\psi^2)$, it follows that

Now, using the facts $X_1|\nu_1=0 \sim N(0,1)$ and $X_1|\nu_1=1 \sim N(0,1+\psi^2)$, we have,
\begin{eqnarray}\label{THM_T1_EB_ALPHA}
\alpha_m
&=& \Pr\big(|X_1| > \sqrt{2\log m} \big) \nonumber\\
% &=& (1-p)\Pr\big(|X_1| > \sqrt{c_1\log m}|\nu_1=0 \big)+p\Pr\big(|X_1| > \sqrt{c_1\log m}|\nu_1=1\big) \nonumber\\
% &=& (1-p)\Pr\big(|Z| > \sqrt{2\log m}\big)+p\Pr\big(|Z| > \sqrt{\frac{2\log m}{1+\psi^2}}\big) \nonumber\\
&=&2\bigg[(1-p)\Pr\big(Z > \sqrt{c_1\log m}\big)+p\Pr\big(Z > \sqrt{\frac{c_1\log m}{1+\psi^2}}\big)\bigg].
\end{eqnarray}
Observe that under the assumption $p\equiv p_m \propto m^{-\epsilon}$, $0 < \epsilon < 1$ and Assumption \ref{ASSUMPTION_ASYMP}, one has $c_1\log m/(1+\psi^2) \rightarrow c_1C/(2\epsilon)$, where the constant $C \in (0,\infty)$ has already been defined in Assumption \ref{ASSUMPTION_ASYMP}. Hence, $\Pr\big(Z > \sqrt{c_1\log m/(1+\psi^2)}\big)=\beta(1+o(1))$, where $\beta=1-\Phi(c_1C/(2\epsilon))$. Clearly, $0 < \beta < 1/2$. On the other hand, applying Mill's ratio, we have $p^{-1}\Pr\big(Z > \sqrt{c_1\log m}\big)\sim\frac{m^{-c_1/2}p^{-1}}{\sqrt{2\pi c_1\log m}} \rightarrow 0$ as $m \rightarrow \infty$. Therefore, using these observations in (\ref{THM_T1_EB_ALPHA}) we obtain $\alpha_m=2\beta p(1+o(1))$.\vspace{1.5mm}

Now, using the definition of the probability of type I error of the $i$-th empirical Bayes decision in (\ref{INDUCED_DECISION_EB}), we have,
\begin{eqnarray}\label{EB_TYPE1_1}
 \widetilde{t}_{1i}
 &=& \Pr\big(E(1-\kappa_i|X_i,\widehat{\tau}) > \frac{1}{2} \big| H_{0i} \mbox{ is true}\big)\nonumber\\
 &=& \Pr\big(E(1-\kappa_i|X_i,\widehat{\tau}) > \frac{1}{2}, \widehat{\tau} \leq 2\alpha_m \big| H_{0i} \mbox{ is true}\big) \nonumber\\
 &&+\Pr\big(E(1-\kappa_i|X_i,\widehat{\tau}) > \frac{1}{2}, \widehat{\tau} > 2\alpha_m \big| H_{0i} \mbox{ is true}\big).
%  && +\Pr\bigg(E(1-\kappa_i|X_i,\widehat{\tau}) > \frac{1}{2}, \widehat{\tau} > 2\alpha_m \bigg| H_{0i} \mbox{ is true}\bigg)
\end{eqnarray}
 As noted in \cite{PKV2014}, for each fixed $x$, $E(1-\kappa_i|x,\tau)$ is non-decreasing in $\tau$. Therefore, $E(1-\kappa_i|X_i,\widehat{\tau}) \leq E(1-\kappa_i|X_i,2\alpha_m)$ whenever $\widehat{\tau} \leq 2\alpha_m$. 
 Thus, we have,
%  \begin{equation}\label{EB_TYPE1_SUBSET}
%   \bigg\{E(1-\kappa_i|X_i,\widehat{\tau}) > \frac{1}{2}, \widehat{\tau} > 2\alpha_m \bigg\} \subseteq \bigg\{E(1-\kappa_i|X_i, 2\alpha_m) > \frac{1}{2} \bigg\}
%  \end{equation}
% %  which implies that
% %  \begin{eqnarray}
% %   \{E(1-\kappa_i|X_i,\widehat{\tau}) > \frac{1}{2}, \widehat{\tau} \leq 2\alpha_m \}
% %   &\subseteq& \{E(1-\kappa_i|X_i, 2\alpha_m) > \frac{1}{2}, \widehat{\tau} \leq 2\alpha_m\}\nonumber\\
% %   &\subseteq& \{E(1-\kappa_i|X_i, 2\alpha_m) > \frac{1}{2}\}\nonumber
% %   \end{eqnarray}
% % %   whence we have,
%  Therefore, using the preceding observation in (\ref{EB_TYPE1_SUBSET}), we obtain,
 \begin{eqnarray}\label{EB_TYPE1_2}
 &&\Pr\big(E(1-\kappa_i|X_i,\widehat{\tau}) > \frac{1}{2}, \widehat{\tau} \leq 2\alpha_m \big| H_{0i} \mbox{ is true}\big)\nonumber\\
  &\leq& \Pr\big(E(1-\kappa_i|X_i, 2\alpha_m) > \frac{1}{2}\big| H_{0i} \mbox{ is true}\big)\nonumber\\
  &\leq& B_{1}^{*}\frac{\alpha_m^{2a}L(\frac{1}{\alpha_m^2})}{\sqrt{\log (\frac{1}{\alpha_m})}}(1+o(1)) \mbox{ as } m\rightarrow\infty,
 %   &&\Pr\bigg(E(1-\kappa_i|X_i,\widehat{\tau}) > \frac{1}{2}, \widehat{\tau} \leq 2\alpha_m \bigg| H_{0i} \mbox{ is true}\bigg)\nonumber\\
% %   &\leq& \Pr\bigg(E(1-\kappa_i|X_i, 2\alpha_m) > \frac{1}{2}, \widehat{\tau} \leq 2\alpha_m \bigg| H_{0i} \mbox{ is true}\bigg)\nonumber\\
%    &\leq& \Pr\bigg(E(1-\kappa_i|X_i, 2\alpha_m) > \frac{1}{2}\bigg| H_{0i} \mbox{ is true}\bigg)\nonumber\\
%   &\leq& B_{1}^{*}\frac{\alpha_m^{2a}L(\frac{1}{\alpha_m^2})}{\sqrt{\log (\frac{1}{\alpha_m})}}(1+o(1)) \mbox{ as } m\rightarrow\infty,
 \end{eqnarray}
 for some finite positive constant $B_{1}^{*}$, independent of $m$. The last step of the above chain of inequalities follows using the same arguments as in the proof of Theorem \ref{THM_T1_UB} and then applying the slowly varying property of $L$.\vspace{1.5mm}
 
 Let $\widehat{\tau}_1 \equiv \frac{1}{m}$ and $\widehat{\tau}_2=\frac{1}{c_2m}\sum_{j=1}^{m}1\{|X_j|> \sqrt{c_1\log m}\}$. Thus, $\widehat{\tau}=\max\{\widehat{\tau}_1,\widehat{\tau}_2\}$.\vspace{1.5mm}
 
 Now observe that, since $\alpha_m=2\beta p(1+o(1))$ and $p \propto m^{-\epsilon}$, for $0 <\epsilon<1$, we have, $1/m < 2\alpha_m$ for all large $m$. Therefore, $\Pr\big(\widehat{\tau}_1 > 2\alpha_m \big| H_{0i} \mbox{ is true}\big)=0$ for all sufficiently large $m$. Using this observation and the fact that $\{\widehat{\tau} > 2\alpha_m\}\subseteq \{\widehat{\tau}_1 > 2\alpha_m\}\bigcup \{\widehat{\tau}_2 > 2\alpha_m\} $, we obtain for all sufficiently large $m$ the following: 
 \begin{eqnarray}\label{EB_TYPE1_3} 
 &&\Pr\big(E(1-\kappa_i|X_i,\widehat{\tau}) > \frac{1}{2}, \widehat{\tau} > 2\alpha_m \big| H_{0i} \mbox{ is true}\big)\nonumber\\
 &\leq& \Pr\big(\widehat{\tau} > 2\alpha_m \big| H_{0i} \mbox{ is true}\big)\nonumber\\
 &\leq& \Pr\big(\widehat{\tau}_1 > 2\alpha_m \big| H_{0i} \mbox{ is true}\big)+\Pr\big(\widehat{\tau}_2 > 2\alpha_m \big| H_{0i} \mbox{ is true}\big)\nonumber\\
 &=& \Pr\big(\widehat{\tau}_2 > 2\alpha_m \big| H_{0i} \mbox{ is true}\big)\nonumber\\
%  &\leq& \Pr\big(\widehat{\tau}_2 > 2\alpha_m, |X_i| > \sqrt{c_1\log m} \big| H_{0i} \mbox{ is true}\big)+\Pr\big(\widehat{\tau}_2 > 2\alpha_m, |X_i| \leq \sqrt{c_1\log m} \big| H_{0i} \mbox{ is true}\big)\nonumber\\
 &\leq& \Pr\big(|X_i| > \sqrt{c_1\log m} \big| H_{0i} \mbox{ is true}\big)+\Pr\big(\widehat{\tau}_2 > 2\alpha_m, |X_i| \leq \sqrt{c_1\log m} \big| H_{0i} \mbox{ is true}\big)\nonumber\\
 &\leq& \frac{1/\sqrt{\pi}}{m^{c_1/2}\sqrt{\log m}} + \Pr\big(\widehat{\tau}_2 > 2\alpha_m, |X_i| \leq \sqrt{c_1\log m} \big| H_{0i} \mbox{ is true}\big)
 \end{eqnarray}
where we use the facts $X_i\sim N(0,1)$ under $H_{0i}$, $1-\Phi(t)< \frac{\phi(t)}{t}$ for $t>0$, and $2/c_1\leq 1$, for the last step in (\ref{EB_TYPE1_3}).\vspace{1.5mm} 

Note that $\widehat{\tau}_2 = \frac{1}{c_2m}\sum_{j(\neq i)=1}^{m}1\{|X_j|> \sqrt{c_1\log m}\}$ over the set $\{|X_i| \leq \sqrt{c_1\log m}\}$. Therefore it follows that,
\begin{eqnarray}\label{EB_TYPE1_4}
 &&\Pr\bigg(\widehat{\tau}_2 > 2\alpha_m, |X_i| \leq \sqrt{c_1\log m} \bigg| H_{0i} \mbox{ is true}\bigg)\nonumber\\
%  &=&\Pr\bigg(\frac{1}{c_2m}\sum_{j(\neq i)=1}^{m}1\{|X_j|> \sqrt{c_1\log m}\} > 2\alpha_m, |X_i| \leq \sqrt{c_1\log m} \bigg| H_{0i} \mbox{ is true}\bigg)\nonumber\\
 &\leq&\Pr\bigg(\frac{1}{c_2m}\sum_{j(\neq i)=1}^{m}1\{|X_j|> \sqrt{c_1\log m}\} > 2\alpha_m\bigg| H_{0i} \mbox{ is true}\bigg)\nonumber\\
 &=&\Pr\bigg(\frac{1}{c_2m}\sum_{j(\neq i)=1}^{m}1\{|X_j|> \sqrt{c_1\log m}\} > 2\alpha_m\bigg)
\end{eqnarray}
where the last step in (\ref{EB_TYPE1_4}) follows from the fact that the distribution of the remaining $X_j$'s do not depend on that of $X_i$. Therefore we have,
\begin{eqnarray}\label{EB_TYPE1_5}
%  &&\Pr\bigg(\widehat{\tau}_2 > 2\alpha_m, |X_i| \leq \sqrt{c_1\log m} \bigg| H_{0i} \mbox{ is true}\bigg)\nonumber\\
 &&\Pr\bigg(\frac{1}{c_2m}\sum_{j(\neq i)=1}^{m}1\{|X_j|> \sqrt{c_1\log m}\} > 2\alpha_m\bigg)\nonumber\\
 &=&\Pr\bigg(\frac{1}{m-1}\sum_{j(\neq i)=1}^{m}1\{|X_j|> \sqrt{c_1\log m}\} > \frac{2c_2m}{m-1}\alpha_m\bigg)\nonumber\\
 &\leq&\Pr\bigg(\frac{1}{m-1}\sum_{j(\neq i)=1}^{m}1\{|X_j|> \sqrt{c_1\log m}\} \geq 2\alpha_m\bigg) \big[\mbox{since } \frac{c_2m}{m-1} > 1 \big].
\end{eqnarray}
Note that $1\{|X_j|> \sqrt{2\log m}\} \stackrel{i.i.d.}{\sim} \mbox{Bernoulli}(\alpha_m)$ for $j \in \{1,\cdots,m\}\setminus\{i\}$. Also $0 < \alpha_m < 2\alpha_m < 1$ for all sufficiently large $m$. Therefore, applying Hoeffding's inequality (\cite{HOEFF1963}), we obtain, for all sufficiently large $m$,
\begin{eqnarray}\label{EB_TYPE1_6}
%  &&\Pr\bigg(\widehat{\tau}_2 > 2\alpha_m, |X_i| \leq \sqrt{c_1\log m} \bigg| H_{0i} \mbox{ is true}\bigg)\nonumber\\
\Pr\bigg(\frac{1}{m-1}\sum_{j(\neq i)=1}^{m}1\{|X_j|> \sqrt{c_1\log m}\} > 2\alpha_m\bigg)
%  &=& \Pr\bigg(\frac{1}{m-1}\sum_{j(\neq i)=1}^{m}1\{|X_j|> \sqrt{c_1\log m}\} - E\big(\frac{1}{m-1}\sum_{j(\neq i)=1}^{m}1\{|X_j|> \sqrt{c_1\log m}\}\big) > \alpha_m\bigg)\nonumber\\
 &\leq& e^{-(m-1)D(2\alpha_m,\alpha_m)}
\end{eqnarray}
where $D(2\alpha_m,\alpha_m)=2\alpha_m \log 2 + (1-2\alpha_m)\log \big(\frac{1-2\alpha_m}{1-\alpha_m}\big)$.\vspace{1.5mm}

% $
% \begin{eqnarray}
%  D(2\alpha_m,\alpha_m)
%  &=& 2\alpha_m \log 2 + (1-2\alpha_m)\log \big(\frac{1-2\alpha_m}{1-\alpha_m}\big)\nonumber
% %  &=& 2\log 2\cdot \alpha_m -(1-2\alpha_m)\log \big(\frac{1}{1- \frac{\alpha_m}{1-\alpha_m}}\big)\nonumber
% \end{eqnarray}
 Recall that $\log (\frac{1}{1-x})/x \rightarrow 1 $ as $ x\downarrow 0$. Therefore, since $\alpha_m \rightarrow 0$ as $m \rightarrow \infty$, one can write $D(2\alpha_m,\alpha_m)$ as,
\begin{eqnarray}\label{EB_TYPE1_7}
 D(2\alpha_m,\alpha_m)
%  &=& 2\alpha_m \log 2 + (1-2\alpha_m)\log \big(\frac{1-2\alpha_m}{1-\alpha_m}\big)\nonumber\\
%  &=& 2\log 2\cdot \alpha_m -(1-2\alpha_m)\log \big(\frac{1}{1- \frac{\alpha_m}{1-\alpha_m}}\big)\nonumber\\
 &=& 2\log 2\cdot \alpha_m -(1-2\alpha_m)\frac{\alpha_m}{1-\alpha_m}\big(1+o(1)\big)\nonumber\\
%  &=& \bigg[ 2\log 2 -\frac{1-2\alpha_m}{1-\alpha_m}\big(1+o(1)\big)\bigg]\alpha_m\nonumber\\
 &=& \big(2\log 2 -1\big)\alpha_m\big(1+o(1)\big).
 \end{eqnarray}
Since $\alpha_m \sim 2\beta p$, we have, $(m-1)D(2\alpha_m,\alpha_m)=2(2\log 2 -1)\beta mp(1+o(1))$ as $m \rightarrow \infty$. Therefore, by combining equations (\ref{EB_TYPE1_1})-(\ref{EB_TYPE1_7}), we finally obtain for each $i=1,\ldots,m$,
 \begin{eqnarray}
  \widetilde{t}_{1i}
  &\leq& B_{1}^{*}\frac{\alpha_m^{2a}L(\frac{1}{\alpha_m^2})}{\sqrt{\log (\frac{1}{\alpha_m^2})}}(1+o(1))+ \frac{1/\sqrt{\pi}}{m^{c_1/2}\sqrt{\log m}} + e^{-2(2\log 2 -1)\beta mp(1+o(1))}\nonumber
 \end{eqnarray}
provided $m$ is sufficiently large. Before we conclude our arguments, it should be noted that all the $o(1)$ terms appeared in the present proof are independent of $i$ and this will be true for any $i=1,\dots,m$. This completes the proof of Theorem \ref{THM_T1_EB_UB}.\vspace{1.5mm}

% Rest of the proof now follows immediately using the fact $\alpha_m\sim 2\beta p$ and the slowly varying property of $L$, with some appropriately chosen constant $A_{1}^{*}> 0$ which is independent of $m$. 

\textbf{{Proof of Theorem \ref{THM_T2_EB_UB}:}}
 We fix any $c_1 \geq 2$ and $c_2 \geq 1$ in the definition of $\widehat{\tau}$ in (\ref{TAU_HAT_EMPB}). Let us choose any fixed $\gamma \in (0,\frac{1}{c_2})$.\vspace{1.5mm}
 
 Now, by the definition of the probability of type II error of the $i$-th decision in the empirical Bayes rule (\ref{INDUCED_DECISION_EB}), we have,
 \begin{eqnarray}\label{EB_TYPE2_1}
  \widetilde{t}_{2i}
 &=& \Pr\big(E(1-\kappa_i|X_i,\widehat{\tau}) \leq \frac{1}{2} \big| H_{1i} \mbox{ is true}\big)\nonumber\\
 &=& \Pr\big(E(\kappa_i|X_i,\widehat{\tau}) \geq \frac{1}{2} \big| H_{1i} \mbox{ is true}\big)\nonumber\\
 &=& \Pr\big(E(\kappa_i|X_i,\widehat{\tau}) \geq \frac{1}{2}, \widehat{\tau} \leq \gamma\alpha_m \big| H_{1i} \mbox{ is true}\big)\nonumber\\ 
 && +\Pr\big(E(\kappa_i|X_i,\widehat{\tau}) \geq \frac{1}{2}, \widehat{\tau} > \gamma\alpha_m \big| H_{1i} \mbox{ is true}\big).
 \end{eqnarray}
 Recall that for each fixed $x \in \mathbb{R}$, $E(\kappa_i|x,\tau)$ is decreasing in $\tau$. Therefore one has $E(\kappa_i|X_i,\widehat{\tau}) \leq E(\kappa_i|X_i,\gamma\alpha_m)$ whenever $\widehat{\tau} > \gamma\alpha_m$, whence it follows that
 \begin{eqnarray}
 \bigg\{E(\kappa_i|X_i,\widehat{\tau}) \geq \frac{1}{2}, \widehat{\tau} > \gamma\alpha_m\bigg\}
%  &\subseteq& \bigg\{E(\kappa_i|X_i,\gamma\alpha_m) \geq \frac{1}{2}, \widehat{\tau} > \gamma\alpha_m\bigg\}\nonumber\\
 &\subseteq& \bigg\{E(\kappa_i|X_i,\gamma\alpha_m) \geq \frac{1}{2}\bigg\}.\nonumber
 \end{eqnarray}
Therefore, applying the same set of arguments used in the proof of Theorem \ref{THM_T2_UB} and using the fact that under $H_{1i}$, $X_i\sim N(0,1+\psi^2)$, we obtain for every fixed $0 <\eta<1/2$ and every fixed $0<\delta<1$, the following:
\begin{eqnarray}\label{EB_TYPE2_2}
 &&\Pr\big(E(\kappa_i|X_i,\widehat{\tau}) \geq \frac{1}{2}, \widehat{\tau} > \gamma\alpha_m \big| H_{1i} \mbox{ is true}\big)\nonumber\\
 &\leq& \Pr\big(E(\kappa_i|X_i,\gamma\alpha_m) \geq \frac{1}{2} \big| H_{1i} \mbox{ is true}\big)\nonumber\\
 &=& \Pr\bigg(|Z| \leq \sqrt{\frac{2a}{\eta(1-\delta)}}\sqrt{\frac{2\log \big(\frac{1}{\gamma\alpha_m}\big)}{1+\psi^2}}\big(1+o(1)\big)\bigg).
\end{eqnarray}
Since $\alpha_m\sim 2\beta p$, using Assumption \ref{ASSUMPTION_ASYMP} it follows that
% Now it should be noted that
\begin{eqnarray}\label{EB_TYPE2_3}
 \frac{2\log \big(\frac{1}{\gamma\alpha_m}\big)}{1+\psi^2}=\frac{-2\log p}{\psi^2}(1+o(1))=C(1+o(1)).
%  &=& \frac{2\log \big(\frac{1}{\gamma\alpha_m}\big)}{\psi^2}(1+o(1))\nonumber\\
%  &=& \frac{1}{\psi^2}\big[-2\log p+ 2\log \big(\frac{1}{2\gamma\beta}\big)-2\log \big(1+o(1)\big)\big](1+o(1))\nonumber\\
%  &=& \frac{-2\log p}{\psi^2}(1+o(1))\nonumber\\
%  &=& C(1+o(1))
\end{eqnarray}
Therefore, from (\ref{EB_TYPE2_2}) and (\ref{EB_TYPE2_3}), for every fixed $0 <\eta<1/2$ and every fixed $0<\delta<1$,  we obtain,
\begin{eqnarray}\label{EB_TYPE2_4}
 &&\Pr\big(E(\kappa_i|X_i,\widehat{\tau}) \geq \frac{1}{2}, \widehat{\tau} > \gamma\alpha_m \big| H_{1i} \mbox{ is true}\big)\nonumber\\
 &\leq& \Pr\bigg(|Z| \leq \sqrt{\frac{2a}{\eta(1-\delta)}}\sqrt{C}\bigg)\big(1+o(1)\big)\nonumber\\
 &=& \bigg[2\Phi\bigg(\sqrt{\frac{2a}{\eta(1-\delta)}}\sqrt{C}\bigg)-1\bigg]\big(1+o(1)\big).
\end{eqnarray}
Our aim is to show now that the first term on the right hand side of (\ref{EB_TYPE2_1}) goes to 0 as $m \rightarrow \infty$. From the definition of $\widehat{\tau}$, it follows that $ \widehat{\tau} \geq \frac{1}{c_2m}\sum_{j(\neq i)=1}^{m}1\{|X_j|> \sqrt{c_1\log m}\}$. Using this observation and noting that the distribution of the remaining $X_j$'s do not depend on the distribution of $X_i$ (because of independence), we obtain the following:
% \begin{eqnarray}
%  \widehat{\tau}
%  &=& \max\big\{\frac{1}{m}, \frac{1}{c_2m}\sum_{j(\neq i)=1}^{m}1\{|X_j|> \sqrt{c_1\log m}\}\big\}\nonumber\\
%  &\geq& \frac{1}{c_2m}\sum_{j(\neq i)=1}^{m}1\{|X_j|> \sqrt{c_1\log m}\}\nonumber\\
%  &\geq& \frac{1}{c_2m}\sum_{j(\neq i)=1}^{m}1\{|X_j|> \sqrt{c_1\log m}\}\nonumber 
% \end{eqnarray}
% 
% Again we have,
\begin{eqnarray}
&&\Pr\big(E(\kappa_i|X_i,\widehat{\tau}) \geq \frac{1}{2}, \widehat{\tau} \leq \gamma\alpha_m \big| H_{1i} \mbox{ is true}\big)\nonumber\\
&\leq& \Pr\big(\widehat{\tau} \leq \gamma\alpha_m \big| H_{1i} \mbox{ is true}\big)\nonumber\\
% &\leq& \Pr\big(\frac{1}{c_2m}\sum_{j(\neq i)=1}^{m}1\{|X_j|> \sqrt{c_1\log m}\} \leq \gamma\alpha_m \big| H_{1i} \mbox{ is true}\big) \nonumber\\
&\leq&\Pr\bigg(\frac{1}{c_2m}\sum_{j(\neq i)=1}^{m}1\{|X_j|> \sqrt{c_1\log m}\} \leq \gamma\alpha_m\bigg) \nonumber\\
% &=& \Pr\big(\frac{1}{m-1}\sum_{j(\neq i)=1}^{m}1\{|X_j|> \sqrt{c_1\log m}\} -\alpha_m \leq -\big(1- \frac{c_2\gamma m}{m-1}\big)\alpha_m\big)\nonumber\\
&=& \Pr\bigg(- \big(\frac{1}{m-1}\sum_{j(\neq i)=1}^{m}1\{|X_j|> \sqrt{c_1\log m}\} -\alpha_m\big) \geq \big(1- \frac{c_2\gamma m}{m-1}\big)\alpha_m\bigg).\nonumber
\end{eqnarray}
Note that $1- \frac{c_2\gamma m}{m-1} > 0$ for all sufficiently large $m$. Therefore, using the preceding arguments and then applying the Markov's inequality, we obtain, for all sufficiently large $m$, the following:
\begin{eqnarray}
&&\Pr\bigg(E(\kappa_i|X_i,\widehat{\tau}) \geq \frac{1}{2}, \widehat{\tau} \leq \gamma\alpha_m \big| H_{1i} \mbox{ is true}\bigg)\nonumber\\
&\leq& \Pr\bigg(- \big(\frac{1}{m-1}\sum_{j(\neq i)=1}^{m}1\{|X_j|> \sqrt{c_1\log m}\} -\alpha_m\big) \geq \big(1- \frac{c_2\gamma m}{m-1}\big)\alpha_m\bigg)\nonumber\\
&\leq& \Pr\bigg(\big|\frac{1}{m-1}\sum_{j(\neq i)=1}^{m}1\{|X_j|> \sqrt{c_1\log m}\} -\alpha_m\big| \geq \big(1- \frac{c_2\gamma m}{m-1}\big)\alpha_m\bigg)\nonumber\\
&\leq& \frac{Var\big(1\{|X_1|> \sqrt{c_1\log m}\}\big)}{(m-1)(1- \frac{c_2\gamma m}{m-1})^2\alpha_m^2}\nonumber\\
&=& \frac{(1-c_2\gamma)^{-2}(1-\alpha_m)}{m\alpha_m}\big(1+o(1)\big)\rightarrow 0 \mbox{ as } m \rightarrow \infty,\nonumber
\end{eqnarray}
whence we have
\begin{equation}\label{EB_TYPE2_5}
 \Pr\bigg(E(\kappa_i|X_i,\widehat{\tau}) \geq \frac{1}{2}, \widehat{\tau} \leq \gamma\alpha_m \bigg| H_{1i} \mbox{ is true}\bigg)=o(1)\mbox{ as } m \rightarrow \infty.
\end{equation}
Combining (\ref{EB_TYPE2_1}), (\ref{EB_TYPE2_4}) and (\ref{EB_TYPE2_5}), it therefore follows that, for each $i=1,\ldots,m$,
we have,
\begin{eqnarray}
\widetilde{t}_{2i}
 &\leq&\bigg[2\Phi\bigg(\sqrt{\frac{2aC}{\eta(1-\delta)}}\bigg)-1\bigg]\big(1+o(1)\big)\nonumber
\end{eqnarray}
for all sufficiently large $m$, where the $o(1)$ term on the right hand side of the above inequality does not depend on $i$ and this will be true for any $i=1,\ldots,m$. This completes the proof of Theorem \ref{THM_T2_EB_UB}. \vspace{1.5mm}
 
\textbf{{Proof of Theorem \ref{THM_BAYES_RISK_UBLB}}} We shall prove only the upper bound here. The corresponding proof for the lower bound will follow analogously. First recall from (\ref{BAYES_RISK_GEN}) the general form of the Bayes Risk of a multiple testing rule under our chosen loss. Now since in our case $t_{1i} =t_1$ and $t_{2i}=t_2$ for all $i=1,\ldots,m$, we have 
  \begin{eqnarray}
  R_{OG} &=& m \big((1-p)t_1+ p t_2 \big) \nonumber \\
          &=& mp \big(\frac{(1-p)}{p}t_1+ t_2). \nonumber 
   \end{eqnarray}       
 To prove the result it suffices to show that under both the situations (I) and (II), $\frac{(1-p)}{p}t_1 \rightarrow 0$ as $m \rightarrow \infty$. We first use the fact that $1-p \leq 1$. Then using the upper bound for $t_1$ obtained in Theorem \ref{THM_T1_UB}, we have
 \begin{equation} \label{TYPE1_BOUND}
 \frac{(1-p)}{p}t_1 \leq \frac{1}{\sqrt{\pi a}} \cdot \frac{2 A_0 K}{a(1-a)} \cdot \frac{\tau}{p} \bigg( \frac{1}{\tau^2}\bigg)^{-(a-1/2)} \frac{L(\frac{1}{\tau^2})}{\sqrt{\log(\frac{1}{\tau^2})}}(1+o(1)) \mbox{ as }m \rightarrow \infty.
 \end{equation}
 The proof under case (I) follows from the facts that $a >1/2$, $\lim_{m\rightarrow\infty} \tau/p \in(0,\infty)$ and hence $\lim_{m \rightarrow \infty} \tau = 0$ and by part (iii) of Lemma \ref{LEM_A_4}, $\lim_{x \rightarrow \infty} x^{-\beta} L(x) = 0$ for any $\beta > 0$ as $L(\cdot)$ is slowly varying. Proof for the case (II) is simple using (\ref{TYPE1_BOUND}).
 \vspace{1.5mm}

\textbf{{Proof of Theorem \ref{THM_BAYES_RISK_EB_UB}}}
Recall that
\begin{equation}\label{THM_3.2_UB1}
  R_{OG}^{EB} = \sum_{i=1}^m\left\{(1-p)\widetilde{t}_{1i}+ p \widetilde{t}_{2i} \right\} = p\sum_{i=1}^m \left\{\frac{1-p}{p}\widetilde{t}_{1i} + \widetilde{t}_{2i}\right\}.
   \end{equation}
For each $i$, the upper bound to $\widetilde{t}_{2i}$ as in Theorem \ref{THM_T2_EB_UB}, is independent of $i$. Therefore, for all sufficiently large $m$,
\begin{equation}\label{THM_3.2_UB2}
 \sum_{i=1}^m \widetilde{t}_{2i}
 \leq m\bigg[2\Phi\bigg(\sqrt{\frac{2aC}{\eta(1-\delta)}}\bigg)-1\bigg]\big(1+o(1)\big)
\end{equation}
Therefore, to complete the proof of the theorem, it will be enough to show
\begin{equation}\label{THM_3.2_UB3}
 \sum_{i=1}^m\frac{1-p}{p} \widetilde{t}_{1i}=o(m) \mbox{ as } m \rightarrow \infty.
\end{equation}
% $\sum_{i=1}^m\frac{1-p}{p} \widetilde{t}_{1i}$ goes to zero as $m\rightarrow\infty$.
For this, we first note that, for each $i=1,\ldots, m$, the upper bound for $\widetilde{t}_{1i}$, as obtained in Theorem \ref{THM_T1_EB_UB}, is independent of $i$. Using this and noting that $1-p<1$, we obtain,
\begin{equation}\label{THM_3.2_UB4}
\frac{1}{m}\sum_{i=1}^m\frac{1-p}{p}\widetilde{t}_{1i} \leq B_{1}^{*}\frac{\alpha_m^{2a}L(\frac{1}{\alpha_m^2})}{p\sqrt{\log (\frac{1}{\alpha_m^2})}}(1+o(1))+ \frac{1/\sqrt{\pi}}{m^{c_1/2}p\sqrt{\log m}} + \frac{1}{p}e^{-2(2\log 2 -1)\beta mp(1+o(1))},
\end{equation}
for all sufficiently large $m$, where $B_{1}^{*}$ and $\beta$ are independent of $m$ and have already been defined in Theorem \ref{THM_T1_EB_UB}.\vspace{1.5mm}

Since $p \rightarrow 0 $ as $m\rightarrow\infty$ and $\alpha_m \sim 2\beta p $, the first term on the right hand side of (\ref{THM_3.2_UB4}) can be shown to go to zero as $m\rightarrow\infty$ under case (I), exactly as in the proof of Theorem \ref{THM_BAYES_RISK_UBLB}, whereas for case (II), it goes to zero as $m\rightarrow\infty$ using the conditions of the theorem. Also, since $p \propto m^{-\epsilon}$ for $0<\epsilon<1$, $mp \rightarrow \infty$ as $m\rightarrow \infty$ and $\log p=o(mp)$. This implies that $\frac{1/\sqrt{\pi}}{m^{c_1/2}p\sqrt{\log m}}$ and $\frac{1}{p}e^{-2(2\log 2 -1)\beta mp(1+o(1))} = e^{-2(2\log 2 -1)\beta mp(1+o(1))-\log p}$ both tend to zero as $m\rightarrow\infty$. These observations, together with (\ref{THM_3.2_UB4}), imply that (\ref{THM_3.2_UB3}) holds, which on combining with (\ref{THM_3.2_UB1}) and (\ref{THM_3.2_UB2}), completes the proof of Theorem \ref{THM_BAYES_RISK_EB_UB}.

\bibliographystyle{apalike}
\bibliography{reference_ba}
\paragraph{Acknowledgement}
We would like to thank Dr. Sourabh Bhattacharya of ISI Kolkata for some useful tips on Markov Chain Monte Carlo simulations. We also thank Professor Jayanta K. Ghosh for letting us aware of the recent work of \cite{PKV2014} on posterior contraction of the horseshoe prior. Research of Malay Ghosh was partially supported by National Science Foundation (NSF) Grants DMS-1007494 and SES-1026165.

\end{document}